\documentclass{amsart}
\usepackage[latin1]{inputenc}
\usepackage{amssymb}
\usepackage{amsmath}
\usepackage{enumerate}
\usepackage{epsfig,color,caption}

\usepackage[dvipsnames]{xcolor}
\usepackage{colonequals, tikz}
\usepackage{tikz-cd}

\usepackage{color}

\usepackage{longtable,graphics,multirow,ulem}
\usepackage{tikz}
\usepackage{pgf,tikz}
\usetikzlibrary{arrows}
\usepackage[all]{xy}

\newcounter{nootje}
\newtheorem{theorem}{Theorem}[section]
\newtheorem{lemma}[theorem]{Lemma}
\newtheorem{proposition}[theorem]{Proposition}
\newtheorem{corollary}[theorem]{Corollary}

\newtheorem{rem}[theorem]{Remark}

\numberwithin{equation}{section}
\newcommand{\rk}{\mbox{rank}}
\newcommand{\NS}{\mbox{NS}}

\newcommand{\MW}{\mbox{MW}}

\newcommand{\ra}{\rightarrow}

\newcommand{\PGL}{\operatorname{PGL}}

\newcommand{\cc}{\mathbb{C}}
\newcommand{\pp}{\mathbb{P}}
\newcommand{\PP}{ \mathbb{P}}
\newcommand{\E}{\mathcal{E}}
\newcommand{\C }{ \mathbb{C}}

\newcommand{\Z}{\mathbb{Z}}

\newcommand{\N}{\mathbb{N}}

\newcommand{\defi}[1]{\textsf{#1}}

\makeatletter
\def\blfootnote{\xdef\@thefnmark{}\@footnotetext}

\title[Elliptic fibrations on covers of the elliptic modular surface]{Elliptic fibrations on covers of the elliptic modular surface of level 5\footnote{ L\MakeLowercase{ast} U\MakeLowercase{pdated:} \today}}

\author{
Francesca Balestrieri
}
\address[Francesca Balestrieri]{
\begin{itemize}
\item[-] University of Oxford, Mathematical Institute, Oxford, OX2 6HD, United Kingdom
\item[-] Partially supported by the EPSRC Scholarship EP$/L505031/1$.
\end{itemize}
} \email[Francesca Balestrieri]{balestrieri@maths.ox.ac.uk}

\author{
Julie Desjardins
}
\address[Julie Desjardins]{
\begin{itemize}
\item[-]
Universit\'e Grenoble Alpes, Institut Fourier, CNRS UMR 5582, 100 rue des Maths, BP 74, 38402 St Martin d'H\`eres, France
\end{itemize}
} \email[Julie Desjardins]{julie.desjardins@imj-prg.fr}

\author{
Alice Garbagnati
}
\address[Alice Garbagnati]{
\begin{itemize}
\item[-]Universit\`a Statale degli Studi di Milano\, Dipartimento di Matematica\\
via Saldini, 50 I20133 Milano
\item[-] Partially supported by FIRB 2012 ``Moduli Spaces and their Applications"
\end{itemize}
} \email[Alice Garbagnati]{alice.garbagnati@unimi.it}

\author{
C\'eline Maistret
}
\address[C\'eline Maistret]{
\begin{itemize}
\item[-] University of Bristol,
Howard House,
Queen's Avenue, Bristol BS8 1SD, UK
\end{itemize}
} \email[C\'eline Maistret]{cm16281@bristol.ac.uk }

\author{
Cec\'ilia Salgado
}
\address[Cec\'ilia Salgado]{
\begin{itemize}
\item[-] Universidade Federal do Rio de Janeiro (UFRJ)\, Instituto de Matem\'atica\\ Cidade Universit\'aria, Ilha do Fund\~ao, Rio de Janeiro.
\item[-] Partially supported by Cnpq grant $446873/2014-4$ and by Faperj Jovem cientista do Nosso estado grant E $10/2016/226621$.
\end{itemize}
} \email[Cec\'ilia Salgado]{salgado@im.ufrj.br}

\author{
Isabel Vogt
}
\address[Isabel Vogt]{
\begin{itemize}
\item[-]
Department of Mathematics, Massachusetts Institute of Technology, 77 Massachusetts Ave, Cambridge, MA 02139, USA
\item[-] Partially supported by National Science Foundation Graduate Research Fellowship Program and grant DMS-1601946.
\end{itemize}
} \email[Isabel Vogt]{ivogt@mit.edu}

\begin{document}

\maketitle

\markleft{}
\begin{abstract}
We consider the K3 surfaces that arise as double covers of the elliptic modular surface of level 5, $R_{5,5}$. Such surfaces have a natural elliptic fibration induced by the fibration on $R_{5,5}$. Moreover, they admit several other elliptic fibrations. We describe such fibrations in terms of linear systems of curves on $R_{5,5}$.  This has a major advantage over other methods of classification of elliptic fibrations, namely, a simple algorithm that has as input equations of linear systems of curves in the projective plane yields a Weierstrass equation for each elliptic fibration. We deal in detail with the cases for which the double cover is branched over the two reducible fibers of type $I_5$ and for which it is branched over two smooth fibers, giving a complete list of elliptic fibrations for these two scenarios.
\end{abstract}
\section{Introduction}

Let $S/\C$ be a smooth projective surface and $B/\C$ be a smooth
projective curve. We say that a proper flat map $\E \colon S \to B$
is an \defi{elliptic fibration} if the generic
fiber $S_b$ is a smooth genus 1 curve and a section
$O\colon B\ra S$
is given. Given a section, we regard the generic fibre of $\E$ as
an elliptic curve over the function field $k(B)$ and so we can work
with a Weierstrass form of $\E$. We will say that an elliptic
fibration is \defi{relatively minimal} if there are no contractible
curves contained in its fibers. For the remainder of this paper all elliptic fibrations will be assumed to be
relatively minimal and not of
product type.

Not all surfaces $S$ admit elliptic fibrations and if $S$ admits an elliptic fibration, a lot is known about the base curve and about the maximal number of elliptic fibrations on it. More precisely if $S$ is of general
type, then $S$ admits no elliptic fibrations;  if the Kodaira
dimension of $S$ is non-positive, then the curve $B$ is rational; if a surface $S$ admits more than one elliptic fibration as above, then it is a K3 surface (a surface with trivial canonical bundle and trivial irregularity).
In particular if $S$ is either a K3 surface or a rational surface,
then $B\simeq \mathbb{P}^1$. Every relatively minimal rational
elliptic surface is the total space of a pencil of plane cubics;
such surfaces admit only the obvious elliptic fibration. We refer to \cite{Mi} and to \cite{SS} for more on the theory of elliptic fibrations on surfaces.

\subsection{K3 surfaces arising from rational elliptic surfaces, and their elliptic fibrations}

In this paper we will consider K3 surfaces $S$ that are double covers of rational elliptic surfaces $R$, branched over two fibers of the elliptic fibration $\E_R:R\ra \mathbb{P}^1$.  More precisely,
such
K3 surfaces are the minimal resolution of the fiber product $\bar{S}$ of a
rational elliptic surface $\E_R\colon R\ra \mathbb{P}^1$ and a degree
two map $\mathbb{P}^1\ra\mathbb{P}^1$, which is necessarily
branched over two points.
We recall that an involution on a
K3 surface is
\defi{non--symplectic} if it acts by $-1$ on $H^0(S, K_S)
\simeq \C$.
Call $\iota$ the involution on $S$ agreeing with the cover involution of
the $2:1$ map $\bar{S}\to  R$; then this involution is non--symplectic and
fixes the inverse image of two fibers of the fibration
$\E_R:R\ra\mathbb{P}^1$, which are the branch curves.
The
quotient
surface $S/\iota$ is rational and is
either $R$
or a
blow-up of $R$ (denoted by $\widetilde{R}$ in what follows).
It was proved by Zhang (see \cite{Z}) that
every K3 surface $S$ admitting a non--symplectic involution
$\iota$ whose fixed locus contains curves of genus at most 1
arises by a base change of order two from a rational elliptic
fibration $\E_R\colon R\ra\mathbb{P}^1$ as described above.

The K3 surface $S$ obtained as above is naturally equipped with one elliptic fibration $\E_S \colon S \to \PP^1$, induced via pullback from $\E_R$.
 In \cite{GS} the
relationship between
elliptic fibrations on
$S$
and linear systems on the rational surface
$R$ (or $\widetilde{R}$) is studied. The elliptic fibrations on $S$ fall into one of the three categories below according to the action of
$\iota$ on the fibers. Note that $\E_S$ belongs to the second one:\\
$\bullet$ if $\iota$ preserves each fiber of the fibration, then
it acts on the fibers as the elliptic involution.  The elliptic
fibration on $S$ is therefore induced by fibrations in rational curves on
$\widetilde{R}$.
We will call these pencils ``conic
bundles" if they are rational fibrations on $R$ and ``generalized
conic bundles" if
they are rational fibrations on $\widetilde{R}$ (but not on $R$);\\
$\bullet$ if $\iota$ preserves the fibration, but not each fiber
of the fibration, this implies that $\iota$ acts on the base of
the fibration (with two fixed points). In this case the elliptic
fibration on $S$ is induced by a pencil of genus 1 curves on $R$,
whose members split in the double cover.  We call these pencils ``splitting genus 1 pencils";\\
$\bullet$ if $\iota$ does not preserve the elliptic fibration, we
call the fibration \defi{of type 3}.
A fibration is of type 3 if and only if the class of the fiber of the fibration, in the N\'eron-Severi group of $S$, is not preserved by $\iota$.


As a result of this classification and of the technique introduced
in \cite{GS}, in good cases one may classify the singular fibers
of all elliptic fibrations $ \E \colon S \to \mathbb{P}^1$  in terms of the
singular fibers of more tractable linear series on $S/\iota$. 
Our focus here is on even finer information: obtaining explicit Weierstrass equations
of elliptic fibrations on such K3 surfaces.  Using Tate's algorithm, one may then read off the singular fibers from the order of vanishing of the coefficients and discriminant of the Weierstrass equation (see, for example, the table on page 41 of \cite{Mi}).
In Sections \ref{subsec: an algorithm (generalized)  conic bundle} and
\ref{subsec:  algorithm splitting genus 1} we give methods and algorithms for determining Weierstrass
equations coming from conic bundles and splitting genus 1 pencils on
rational elliptic surfaces, under some assumptions.  This is our main result.
\begin{theorem}
Let $S$ be a K3 surface arising from the rational elliptic surface
$\E_R\colon R\ra\mathbb{P}^1$ as described above. Let $\E$ be an
elliptic fibration on $S$ that is not of type 3. Then, under certain conditions, one obtains
a Weierstrass equation for $\E$ by applying
\begin{itemize}
\item the Algorithm of Section \ref{subsec: an algorithm (generalized)  conic bundle}  if $\E$ is induced by a (generalized) conic bundle with prescribed properties (see Section 5.2 for details);
\item the Algorithm of Section
\ref{subsec:  algorithm splitting genus 1} if $\E$ is induced by a splitting genus 1 pencil.
\end{itemize}
\end{theorem}

\subsection{Outline of the paper}

We focus particularly on K3 surfaces arising as double covers of
$R_{5,5}$, the elliptic modular surface of level 5. This is the
universal elliptic curve over the modular curve $X_1(5)$ and the
evident map $\E_{R_{5,5}} \colon R_{5,5} \to X_1(5) \simeq \PP^1$
is the unique elliptic fibration.  The fibers of $\E_{R_{5,5}}$
are smooth except for two nodal rational curves (type $I_1$) and
two $5$-gons (type $I_5$); this property also determines $R_{5,5}$
and implies that the Mordell-Weil group $\MW(\E_R)=\Z/5\Z$.  The geometry of this surface
as the total space of a pencil of plane cubics is described in
Section \ref{sec: R55}.

In Section \ref{sec: conic bundles} we
classify the conic bundles on $R_{5,5}$ up to automorphisms.
The main result of this Section is Proposition
\ref{prop: classification conic bundles}.


In the remainder of the paper we study elliptic fibrations on different K3 surfaces arising from $R_{5,5}$ by choosing different base changes.
In Section  \ref{sec: K3 covers of R55} we describe such K3 surfaces,
writing their equations as double covers of $\mathbb{P}^2$, as well as giving the Weierstrass form of the elliptic fibrations induced
by $\E_R$.

In Section \ref{sec:
elliptic fibrations induced by conic bundles} we consider the
elliptic fibrations induced on K3 surfaces by the conic bundles classified in Section \ref{sec: conic bundles}.
We
observe directly that the same conic bundle induces elliptic
fibrations with very different properties on K3 surfaces
according to the choice of the branch curves.
In Sections \ref{sec: the K3 surface S55} and \ref{sec: the K3 surface X} we restrict our attention to two K3
surfaces obtained by choosing maximally different branch fibers:
in Section \ref{sec: the K3 surface S55} we consider the very special case where the branch fibers are $2I_5$ and in Section \ref{sec: the K3 surface X} we consider the generic case where the branch fibers are $2I_0$.

When the double cover is branched over the two $5$-gons, the K3
surface is called $S_{5,5}$ and the involution $\iota$ fixes the
union of $10$ rational curves. There is a unique such K3 surface
possessing such a non--symplectic involution. This K3 surface
admits 13 types of elliptic fibrations, classified by Nishiyama in
\cite{Nish}. In this special case we are able to determine
equations for all elliptic fibrations on $S_{5,5}$ using our
techniques and algorithms. We observe that in this case there are
no fibrations of type 3. The main result of this
Section is Proposition \ref{prop: classification elliptic
fibrations S55}.

When the double cover is branched on two smooth fibers, the K3
surface moves in the 2-dimensional family of the K3 surfaces
admitting an elliptic fibration with a 5-torsion section. We call
a very general member of this family $X_{5,5}$ and using the
lattice-theoretic technique of \cite{Nish} we list all the
admissible configurations of fibers of an elliptic fibration on
$X_{5,5}$ in Table \ref{eq: table elliptic fibrations on X55}
proving Proposition \ref{prop: classification
elliptic fibrations on X55}. In this case the elliptic fibrations
cannot be induced by splitting genus 1 pencils or by generalized
conic bundles.
On the other hand there are plenty
of elliptic fibrations of type 3 and we describe one of them in
detail.

\subsection*{Acknowledgements}

This work was initiated during the Women in Numbers Europe 2 workshop.  We thank the organizers and the Lorentz Center in Leiden for providing such a stimulating research environment.  We also thank Bernd Sturmfels for drawing our attention to the connections with tropical geometry, and the referee for many helpful suggestions and comments that have improved this article.

\section{The surface $R_{5,5}$}\label{sec: R55}

Let $R_{5,5}$ be the elliptic modular surface of level 5.
We know from \cite[Tableau]{Beau} (with coordinates $X = x_0, Y = x_0-x_1, Z = x_2$) that the surface $R_{5,5}$ is the blow-up of $\mathbb{P}^2$ in the
basepoints of the pencil $\mathcal{P}$ of cubics
\begin{equation}\label{eq: pencil of cubics R5511}\lambda x_1x_2(x_0-x_1)+\mu
x_0(x_0-x_1-x_2)(x_0-x_2)=0.\end{equation} We will denote this blow-up morphism by $\beta \colon R_{5,5} \to \mathbb{P}^2$.  The cubic corresponding
to $\lambda=0$ is reducible and consists of the three lines
$m_1: {x_0=0}$,  $m_2: {x_0-x_1-x_2=0}$, $m_3: {x_0-x_2=0}$; the
cubic corresponding to $\mu=0$ is reducible and consists of the
three lines $\ell_1: {x_0-x_1=0}$, $\ell_2: {x_1=0}$, $\ell_3: {x_2=0}$.
\begin{figure}[ht]

\begin{tikzpicture}[scale=.5]


\draw (1,12)node[above right]{$\ell_1$}  -- (-5,0) ;

\draw (-1,12)node[above left]{$\ell_2$}  -- (4,2) ;

\draw (-5,4)--(5,4)node[right]{$\ell_3$};

\draw (-4/3,2)node[below]{$m_1$} -- (1/3,12);

\draw (-5,4-4/5) -- (4,6 + 4/5)node[below right]{$m_2$};

\draw (-5,4/3) -- (5,8)node[above right]{$m_3$};


\draw (3,4) node{$\bullet$};
\draw (3,4) node[above right]{$T_1$};

\draw (-6/7, 34/7) node{$\bullet$};
\draw (-6/7, 34/7) node[above left]{$T_2$};


\draw (0,10) node{$\bullet$};
\draw (0,10) node[left]{$Q_1$};

\draw (-3,4) node{$\bullet$};
\draw (-3,4) node[above left]{$Q_2$};

\draw (-1,4) node{$\bullet$};
\draw (-1,4) node[below right]{$Q_3$};

\draw (2,6) node{$\bullet$};
\draw (2,6) node[below right]{$Q_4$};

\draw (-4,2)node{$\bullet$};
\draw (-4,2)node[below right]{$Q_5$};

\end{tikzpicture}

\caption{The reducible cubics and basepoints of the pencil $\mathcal{P}$.}
\end{figure}
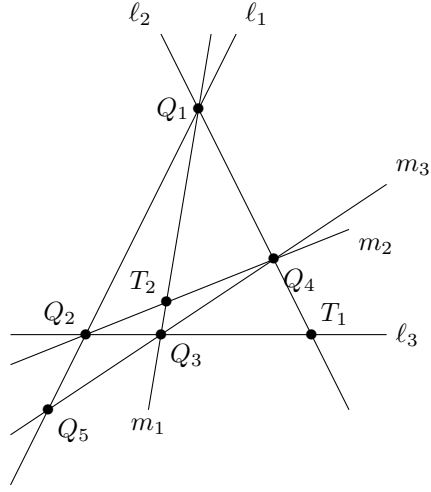

The nine basepoints of this pencil of cubics are: $Q_1:=(0:0:1)$;
the point $Q_1'$, infinitely near to $Q_1$ and corresponding to
the tangent direction $m_1$;  $Q_2:=(1:1:0)$; the point $Q_2'$,
infinitely near to $Q_2$ and corresponding to the tangent
direction $m_2$; $Q_3:=(0:1:0)$; the point $Q_3'$, infinitely near
to $Q_3$ and corresponding to the tangent direction $\ell_3$;
$Q_4:=(1:0:1)$; the point $Q_4'$, infinitely near to $Q_4$ and
corresponding to the tangent direction $\ell_2$; $Q_5:=(1:1:1)$.

In the following we will denote by $T_1$ the point $(1:0:0)$,
which is not a basepoint of the pencil and corresponds to the
intersection of the lines $\ell_2$ and $\ell_3$, and by $T_2$ the point
$(0:1:-1)$, which is not a basepoint of the pencil and
corresponds to the intersection of the lines $m_1$ and $m_2$.

Let $h$ denote the preimage of the class of a line; then $\NS(R_{5,5})$ is spanned by $h$ and the components of the exceptional divisors of the blow up $\beta:R_{5,5}\ra\mathbb{P}^2$.  We will denote the (irreducible) exceptional divisor corresponding to $Q_i$ (resp. $Q_i'$) by $E_i$ (resp. $F_i$) for $i = 1, 2, 3, 4$.  At $Q_5$ there is only $E_5$.  Note that $F_i^2 = -1$, $E_i^2 = -2$ for $i= 1, 2, 3, 4$, $E_5^2=-1$, $E_iE_j=0$ if $i\neq j$, $F_iF_j=0$ if $i\neq j$, $E_iF_i=1$. By slight abuse of notation, let $\ell_1, \ell_2, \ell_3$ and $m_1, m_2, m_3$ denote the proper transforms on $R_{5,5}$ of the corresponding lines in $\PP^2$.
We have the following relations in $\NS(R_{5,5})$:
\begin{align}\begin{array}{lcl}
\ell_1=h-E_1-F_1-E_2-F_2-E_5&& m_1=h-E_1-2F_1-E_3-F_3\\
\ell_2=h-E_1-F_1-E_4-2F_4&& m_2=h-E_2-2F_2-E_4-F_4\\
\ell_3=h-E_1-F_1-E_3-2F_3&& m_3=h-E_3-F_3-E_4-F_4-E_5.\\
\end{array}
\end{align}

The Weierstrass equation of the elliptic fibration of $R_{5,5}$ is obtained by \eqref{eq: pencil of cubics R5511} choosing $x_2=1$ and applying standard transformations. It is
\begin{equation}\label{eq: Weierstrass R5511} y^2=x^3+A(\lambda:\mu)x+B(\lambda:
\mu), \ \ \mbox{where}\end{equation}
 $$A(\mu):=\frac{-1}{48}\mu^4-\frac{1}{4}\mu^3\lambda-\frac{7}{24}\mu^2\lambda^2+\frac{1}{4}\mu\lambda^3-\frac{1}{48}\lambda^4,\mbox{  and }$$
$$B(\mu):=\frac{1}{864}\mu^6+\frac{1}{48}\mu^5\lambda+\frac{25}{288}\mu^4\lambda^2+\frac{25}{288}\mu^2\lambda^4-\frac{1}{48}\mu\lambda^5+\frac{1}{864}\lambda^6. $$
The discriminant is $\frac{1}{16}\mu^5\lambda^5(-\lambda^2+11\mu\lambda+\mu^2)$, so there are two fibers of type $I_5$ over $(\lambda:\mu)=(0:1)$ and $(\lambda:\mu)=(1:0)$. Moreover there are two fibers of type $I_1$ over $(\lambda:\mu)=(1:-\frac{11}{2}\pm\frac{5}{2}\sqrt{5})$.

Now the function
$$\mu\mapsto (x(\lambda:\mu);y(\lambda:\mu)) = \left(\frac{\mu^2-6\mu\lambda+\lambda^2}{12};\frac{\mu^2\lambda}{2}\right)$$ is a
5-torsion section of this fibration. It is known, see e.g.\ \cite[Section 9.5]{SS} that
the elliptic fibration on $R_{5,5}$ has Mordell--Weil group equal
to $\Z/5\Z$.

The negative curves on $R_{5,5}$ are
\begin{enumerate}
\item the ten components of the two fibers of type $I_5$ denoted by $\Theta_0^{(1)}$, $\Theta_1^{(1)}$, $\Theta_2^{(1)}$, $\Theta_3^{(1)}$, $\Theta_4^{(1)}$ on the first fiber and $\Theta_0^{(2)}$, $\Theta_1^{(2)}$, $\Theta_2^{(2)}$, $\Theta_3^{(2)}$, $\Theta_4^{(2)}$ on the second fiber (these are all $(-2)$-curves);
\item the five sections $P_0$, $P_1$, $P_2$, $P_3$, and $P_4$, where $P_0$ meets the components $\Theta_0^{(1)}$ and $\Theta_0^{(2)}$, $P_1$ meets the components $\Theta_1^{(1)}$ and $\Theta_2^{(2)}$, $P_2$ meets the components $\Theta_2^{(1)}$ and $\Theta_4^{(2)}$, $P_3$ meets the components $\Theta_3^{(1)}$ and $\Theta_1^{(2)}$ and $P_4$ meets the components $\Theta_4^{(1)}$ and $\Theta_3^{(2)}$ (these sections are all $(-1)$-curves).
\end{enumerate}

The dual graph of this configuration is given in Figure \ref{negative_curves_R5511}. We observe that the Figure \ref{negative_curves_R5511} is a generalization of the Petersen graph (it is exactly the Petersen graph if one does not consider the empty edges) and we point out that this graph represents several intersecting objects in algebraic geometry and in tropical geometry, see eg.\ \cite{RSS}.
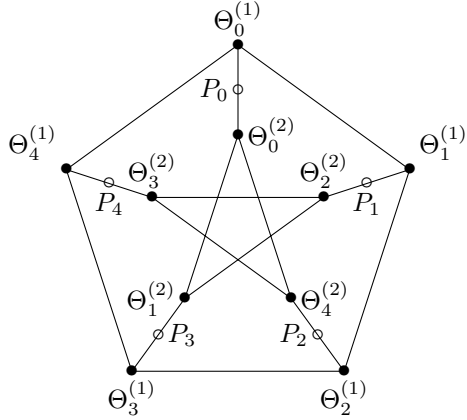
\begin{figure}[h]
\begin{tikzpicture}[scale=.6]

\draw (0,4*1) node{$\bullet$};

\draw (4*0.951, 4*0.309) node{$\bullet$};

\draw (4*0.588, 4*-0.809)node{$\bullet$};

\draw (4*-0.588, 4*-0.809)node{$\bullet$};

\draw (4*-0.951, 4*0.309) node{$\bullet$};

\draw (0,3*1) node{$\circ$};

\draw (3*0.951, 3*0.309) node{$\circ$};

\draw (3*0.588, 3*-0.809)node{$\circ$};

\draw (3*-0.588, 3*-0.809)node{$\circ$};

\draw (3*-0.951, 3*0.309) node{$\circ$};

\draw (0,2*1) node{$\bullet$};

\draw (2*0.951, 2*0.309) node{$\bullet$};

\draw (2*0.588, 2*-0.809)node{$\bullet$};

\draw (2*-0.588, 2*-0.809)node{$\bullet$};

\draw (2*-0.951, 2*0.309) node{$\bullet$};

\draw (0,4*1) --  (4*0.951, 4*0.309) --(4*0.588, 4*-0.809) -- (4*-0.588, 4*-0.809) -- (4*-0.951, 4*0.309) --(0,4);

\draw (0,2*1)--(2*0.588, 2*-0.809) -- (2*-0.951, 2*0.309)  --  (2*0.951, 2*0.309) -- (2*-0.588, 2*-0.809) --(0,2);

\draw (0,4)node[above]{$\Theta_0^{(1)}$} -- (0,3)node[left]{$P_0$} -- (0,2)node[right]{$\Theta_0^{(2)}$};

\draw (4*0.951, 4*0.309)node[above right]{$\Theta_1^{(1)}$} -- (3*0.951, 3*0.309)node[below]{$P_1$}-- (2*0.951, 2*0.309)node[above]{$\Theta_2^{(2)}$};

\draw (4*0.588, 4*-0.809)node[below]{$\Theta_2^{(1)}$} -- (3*0.588, 3*-0.809)node[left]{$P_2$} --(2*0.588, 2*-0.809)node[right]{$\Theta_4^{(2)}$};

\draw (4*-0.588, 4*-0.809)node[below]{$\Theta_3^{(1)}$} -- (3*-0.588, 3*-0.809)node[right]{$P_3$} -- (2*-0.588, 2*-0.809)node[left]{$\Theta_1^{(2)}$};

\draw (4*-0.951, 4*0.309)node[above left]{$\Theta_4^{(1)}$} --(3*-0.951, 3*0.309)node[below]{$P_4$} -- (2*-0.951, 2*0.309)node[above]{$\Theta_3^{(2)}$};

\end{tikzpicture}
\caption{Dual graph of negative curves on $R_{5,5}$.  The symbol $\bullet$ denotes a $(-2)$-curve, and $\circ$ denotes a $(-1)$-curve.}\label{negative_curves_R5511}
\end{figure}

We may then make the following choice of identifications:
\begin{eqnarray}\label{R_5,5 curve ids}
\begin{array}{c c c c c}
\Theta_0^{(1)} = m_1&\Theta_1^{(1)} = E_3&\Theta_2^{(1)} = m_3 &\Theta_3^{(1)} = E_4 &\Theta_4^{(1)} = m_2\\
 \Theta_0^{(2)} = E_1  & \Theta_1^{(2)} = \ell_2& \Theta_2^{(2)} = \ell_3 & \Theta_3^{(2)} = E_2  & \Theta_4^{(2)} = \ell_1\\
 P_0 = F_1  & P_1 = F_3&P_2 = E_5& P_3 = F_4&P_4 = F_2.
\end{array}
\end{eqnarray}


We observe that there is  an automorphism $\sigma_5$ on $R_{5,5}$ of order 5,  which is the translation by the section $P_1$. It acts on the negative curves as follows: $\sigma_5(\Theta_i^{(1)})=\Theta_{i+1}^{(1)}$ and $\sigma_5(\Theta_i^{(2)})=\Theta_{i+2}^{(2)}$, where $i+1$ and $i+2$ are considered modulo $5$; $\sigma_5(P_k)=P_{k+1}$, where $k+1$ is considered modulo $5$.

There is also an automorphism $\sigma_2$ of order 2 on $R_{5,5}$
which is the elliptic involution on the elliptic curve \eqref{eq: Weierstrass R5511}
over the function field $k(\mu)$.
Note that $\sigma_2$
restricts to the elliptic involution on each smooth fiber of the fibration \eqref{eq: Weierstrass R5511}. It acts on the negative curves as follows: $\sigma_2(\Theta_i^{(j)})=\Theta_{-i}^{(j)}$, where $i\in\Z/5\Z$, $j=1,2$ and $\sigma_2(P_k)=P_{-k}$, where $k\in\Z/5\Z$.

There is also an automorphism $\alpha$ of $R_{5,5}$ lying above the involution of $\mathbb{P}^1$ $\alpha\colon (\lambda:\mu)\mapsto(-\mu:\lambda)$.  In terms of the Weierstrass equation \eqref{eq: Weierstrass R5511}, we have
\[\alpha \colon (x, y, \mu) \mapsto (x/\mu^2, -y/\mu^3, -1/\mu). \]
Note that the automorphism $\sigma_5^2 \alpha \sigma_5^3$ is induced by the element 
$\left(\begin{smallmatrix} 0 & 1 & 0 \\ -1 & 1 & 1 \\ 1 & 0 & 0 \end{smallmatrix}\right) \in \PGL_3(\C)$.  From this description, the action on $\NS(R_{5,5})$ is apparent: $\alpha(\Theta^{(1)}_0) = \Theta^{(2)}_0$ and $\alpha(\Theta^{(2)}_0) = \Theta^{(1)}_0$; $\alpha(\Theta^{(1)}_i) = \Theta^{(2)}_i$ and $\alpha(\Theta^{(2)}_i) = \Theta^{(1)}_{-i}$; and finally $\Theta(P_0) = P_0$ while the remaining sections are permuted so as to preserve intersections.

 
\section{Conic bundles on $R_{5,5}$}\label{sec: conic bundles}

In this section we classify the conic bundles on $R_{5,5}$ by considering their reducible fibers proving Proposition \ref{prop: classification conic bundles}

The key result we use is that on a rational elliptic surface, every conic
bundle has at least one reducible fiber. Further, any reducible
fiber must be of type $A_n$ or $D_m$, as shown in Figure
\ref{RES_conic_reducible_fibers}, see e.g. \cite{GS}.

\begin{figure}[h]\label{RES_conic_reducible_fibers}
\begin{tikzpicture}

\draw (0,0) node{$\circ$};
\draw (0,0) -- (1,0)--(2,0);
\draw (1,0) node{$\bullet$};
\draw (2,0) node{$\bullet$};
\draw[dashed] (2,0) -- (3,0);
\draw (3,0) node{$\bullet$};
\draw (4,0) node{$\circ$};
\draw (3,0) --(4,0);

\draw (2,-1) node{$A_n$};

\draw (6,0.5) node{$\bullet$};
\draw (6,-0.5) node{$\bullet$};
\draw (6,0.5) -- (7,0)--(8,0);
\draw (6,-0.5) -- (7,0);
\draw (7,0) node{$\bullet$};
\draw (7,0) node[above]{\small{2}};
\draw (8,0) node{$\bullet$};
\draw (8,0) node[above]{\small{2}};
\draw[dashed] (8,0) -- (9,0);
\draw (9,0) node{$\bullet$};
\draw (9,0) node[above]{\small{2}};
\draw (10,0) node{$\circ$};
\draw (10,0) node[above]{\small{2}};
\draw (9,0) --(10,0);

\draw (8,-1) node{$D_m$};
\end{tikzpicture}
\caption{Possible reducible fibers of conic bundles on (minimal) rational elliptic surfaces.  The number $n$ and $m$ refer to the number of components.  Multiplicity of a component is indicated above the corresponding vertex if it is not 1.   The symbol $\bullet$ denotes a $(-2)$-curve, and $\circ$ denotes a $(-1)$-curve.}\label{RES_conic_reducible_fibers}
\end{figure}
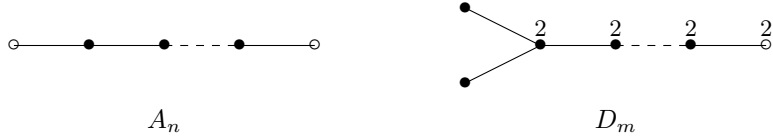

\begin{proposition}\label{prop: classification conic bundles} There are exactly three conic bundles on $R_{5,5}$ up to automorphisms, i.e. the conic bundles $B_1$, $B_2$ and $B_3$ induced by the pencils of plane rational curves with equations \ref{eq: B_1}, \ref{eq: B_2} and \ref{eq: B_3} respectively.
\end{proposition}

\proof {\bf Step 1: classification of the reducible fibers.} Every
conic bundle has at least one reducible fiber, so in order to
classify the conic bundles it suffices to find all the possible
reducible fibers. The components of the reducible fibers are
negative curves. As $R_{5,5}$ (and in fact any extremal rational
elliptic surface, see \cite[VIII.1.2]{Mi}) has only finitely many curves of negative
self-intersection, one simply must find all possible $A_n$ and
$D_m$ (for $m \ge 3$) configurations among the curves with
negative self intersection.

Every reducible fiber contains at least a $(-1)$-curve, as shown
in Figure \ref{RES_conic_reducible_fibers}. Since we are looking
for a classification up to automorphisms, and the automorphism
$\sigma_5$ permutes the $(-1)$-curves on $R_{5,5}$, we can always
assume that one of the reducible fibers of the conic bundle
contains the $(-1)$-curve $P_0$. Moreover we recall that the
action of $\sigma_2$ switches $\Theta_i^{(j)}$ with
$\Theta_{5-i}^{(j)}$ and the action of $\alpha$ switches the two $I_5$-fibers. Up to the action of $\sigma_2$ and $\alpha$, the
reducible fibers of type $A_n$ that contain $P_0$ as a component
are necessarily of one of
the following type:\\
$\bullet$ $P_0+\sum_{i=0}^k\Theta_i^{(1)}+P_k$, $k=1,2,3$ and in
this case we have a fiber of type $A_{k+3}$.\\

Up to the action of $\sigma_2$ the reducible fibers of type $D_m$
that contain $P_0$ as a
component are necessarily of the following types:\\
$\bullet$ $2P_0+\Theta_0^{(1)}+\Theta_0^{(2)}$ and in
this case we have a fiber of type $D_3$;\\
$\bullet$ $2P_0+2\Theta_0^{(1)}+\Theta_1^{(1)}+\Theta_4^{(1)}$ and
in
this case we have a fiber of type $D_4$.\\

{\bf Step 2: at most three conic bundles.} We consider the conic
bundles associated to the possible singular fibers described above
and show that some of them are equivalent up to the action of
$\sigma_5$ and $\sigma_2$. Moreover we describe the conic bundles
that we find and we gives equations for them:\\

$(1)$ $\vert B_1\vert = \vert P_0 + \Theta_0^{(1)} +
\Theta_1^{(1)} + P_1\vert$. The components of one reducible fiber
are $P_0$, $\Theta_0^{(1)}$, $\Theta_1^{(1)}$, $P_1$. This fiber
is of type $A_4$. The curves $\Theta_2^{(1)}$, $\Theta_4^{(1)}$,
$\Theta_0^{(2)}$ and $\Theta_2^{(2)}$ are sections of the bundle.
There is another reducible fiber of type $A_4$ which is formed by
the curves $P_2$, $\Theta_4^{(2)}$, $\Theta_3^{(2)}$, $P_4$, and
one of type $D_3$ which is formed by $P_3$ (with multiplicity 2),
$\Theta_3^{(1)}$, $\Theta_1^{(2)}$.

Using the identifications made earlier, this class can also be written as:
\begin{align*}
B_1 &= F_1 + m_1 + E_3 + F_3 = F_1 + (h - E_1 - 2F_1 - E_3 - F_3) + E_3 + F_3 = h - E_1 - F_1.
\end{align*}
Unwinding what this means geometrically: $\vert B_1 \vert$ comes via proper transform from the pencil of lines through $Q_1$ in $\PP^2$ which has equation \begin{equation}\label{eq: B_1}
x_1=\tau x_0.\end{equation} Under this description we can also understand the singular fibers: they correspond to the special lines $m_1, \ell_1$, and $\ell_2$.  For example the line $\ell_2$ corresponds to the reducible fiber $\beta^*(\ell_2) -E_1 - F_1 = \ell_2 + E_4 + 2F_4 = \Theta_1^{(2)} + \Theta_3^{(1)} + 2P_3$.

The conic bundle $\vert B_1\vert$ is sent to other conic bundles
by $\sigma_5$, by $\sigma_2$ and by their powers. Each of these
has exactly three reducible fibers of types $A_4$, $A_4$ and
$D_3$.

We observe that the fiber of type $D_3$ of the conic bundle
$|B_1|$ is sent to $2P_0+\Theta_0^{(1)}+\Theta_2^{(2)}$ by the
automorphism $\sigma_5^2$, so the conic bundle with reducible
fiber $2P_0+\Theta_0^{(1)}+\Theta_0^{(2)}$ is equivalent to
$|B_1|$ up to automorphisms.

Similarly the $A_4$-fiber $P_2+\Theta_4^{(2)}+\Theta_3^{(2)}+P_4$
is sent to $P_3+\Theta_1^{(2)}+\Theta_0^{(2)}+P_0$ by
$\sigma_5$, so also the conic bundle with reducible fiber
$P_0+\Theta_0^{(2)}+\Theta_1^{(2)}+P_3$ is equivalent to $|B_1|$
up to automorphisms.

$(2)$ $\vert B_2 \vert=\vert
P_0+\Theta_0^{(1)}+\Theta_1^{(1)}+\Theta_2^{(1)}+P_2\vert$.  The
components of one reducible fiber are $P_0$, $\Theta_0^{(1)}$,
$\Theta_1^{(1)}$, $\Theta_2^{(1)}$, $P_2$. This fiber is of type
$A_5$. The curves $\Theta_3^{(1)}$, $\Theta_4^{(1)}$,
$\Theta_0^{(2)}$, $\Theta_4^{(2)}$ and $P_1$ are sections of the
bundle. There is another reducible fiber of type $A_5$ which is
formed by the curves $P_3$, $\Theta_1^{(2)}$, $\Theta_2^{(2)}$,
$\Theta_3^{(2)}$, $P_4$.

Similarly here we can write:
\begin{align*}
B_2 &= F_1 + m_1 + E_3 +m_3 +E_5 
= 2h - E_1 - F_1 - E_3 - 2F_3 - E_4 - F_4.
\end{align*}
Hence $B_2$ corresponds to the pencil of conics through $Q_1, Q_3, Q_3'$, and $Q_4$.  More explicitly this is given by conics passing through $Q_1$ and $Q_4$, and tangent to $\ell_3$ at $Q_3$ and in $\PP^2$ this pencil is given by the equation \begin{equation}\label{eq: B_2} x_1x_2=\tau ( x_0x_2-x_0^2).
\end{equation} The two reducible fibers correspond to the reducible conics $m_1 \cup m_3$ and $\ell _2 \cup \ell_3$.

The $A_5$-fiber of $|B_2|$ whose components are $P_3$,
$\Theta_1^{(2)}$, $\Theta_2^{(2)}$, $\Theta_3^{(2)}$, $P_4$ is
sent to the reducible fiber $P_0+\sum_{i=0}^2 \Theta_i^{(2)}+P_1$
by $\sigma_5^2$.

$(3)$ $\vert B_3\vert=\vert
P_0+\Theta_0^{(1)}+\Theta_1^{(1)}+\Theta_2^{(1)}+\Theta_3^{(1)}+P_3\vert$.
The components of one reducible fiber are $P_0$, $\Theta_0^{(1)}$,
$\Theta_1^{(1)}$, $\Theta_2^{(1)}$, $\Theta_3^{(1)}$, $P_3$. This
fiber is of type $A_6$. The curves $\Theta_0^{(2)}$,
$\Theta_1^{(2)}$, $P_1$ and $P_2$ are sections of the bundle. The
curve $\Theta_4^{(1)}$ is a multisection of degree $2$.  There is
another reducible fiber of type $D_4$ which is formed by the
curves $P_4$, $\Theta_3^{(2)}$, $\Theta_4^{(2)}$,
$\Theta_2^{(2)}$.

We can also describe this using:
\begin{align*}
B_3 &= F_1 + m_1 + E_3 +m_3 +E_4 + F_4
= 2h - E_1 - F_1 - E_3 - 2F_3 - E_5.
\end{align*}
Therefore $B_3$ comes from the pencil of conics in $\PP^2$ through
$Q_1, Q_3, Q_3'$ and $Q_5$; that is conics through $Q_1$ and
$Q_5$, tangent to $\ell_3$ at $Q_3$, and in $\PP^2$ this pencil is
given by the equation
\begin{equation}\label{eq: B_3}x_1x_2=(\tau+1) x_0x_2-\tau x_0^2.\end{equation} We can again understand the reducible fibers as coming from reducible conics $\ell_1 \cup \ell_3$ and $m_1 \cup m_3$.

The $D_4$-fiber is sent to the fiber
$2P_0+2\Theta_0^{(2)}+\Theta_1^{(2)}+\Theta_4^{(2)}$ by
$\sigma_5$.





{\bf Step 3: exactly three conic bundles.} It remains only to prove
that the conic bundles $B_i$ for $i=1,2,3$ are all inequivalent
up to automorphisms. Since the reducible fibers of $\vert
B_1\vert$ are ($2A_4,D_3$), the reducible fibers of $\vert
B_2\vert$ are ($2A_5$) and the reducible fibers of $\vert
B_3\vert$ are ($A_6,D_4$), they cannot be equivalent up to
automorphisms. Hence there are three
conic bundles on $R_{5,5}$ up to automorphisms.\endproof

\section{K3 surfaces obtained by $R_{5,5}$}\label{sec: K3 covers of R55}

Now we consider K3 surfaces obtained from $R_{5,5}$ by a base
change of order 2 branched on two fibers. Of course the K3 surface
obtained depends on the branch fibers. Let us explicitly give the
description of the K3 surfaces that we can obtain in this way.
They will be both described as elliptic fibrations (induced by the
one of $R_{5,5}$) and as double covers of $\mathbb{P}^2$.

\subsection{The branch fibers are $2I_5$: the K3 surface $S_{5,5}$}

Let us consider the K3 surface $S_{5,5}$ obtained by a base change of order 2 of $R_{5,5}$ whose branch locus corresponds to the two fibers of type $I_5$. This means that all the components of the fibers of type $I_5$ are in the branch locus of the double cover $S_{5,5}\dashrightarrow R_{5,5}$.

\subsubsection{The surface $\widetilde{R}$}

The double cover of $R_{5,5}$ branched over the two fibers of type $I_5$ has ordinary double point singularities at the 10 points over the nodes in the branch fibers.  In order to obtain a K3 surface one can blow-up these $10$ points on the double cover, introducing 10 exceptional divisors.   Equivalently one may first blow-up the $10$ nodes of the branch fibers to obtain a non-minimal rational elliptic surface $\widetilde{R}$ and then normalize the double cover of this surface branched over the preimage of the branch fibers.  Note that in the preimage the $10$ exceptional curves all occur with multiplicity $2$, and so they are not in the branch locus after normalization.


We will make use of this non-minimal rational elliptic surface $\widetilde{R}$; it is simply the blow-up of $\mathbb{P}^2$ in $9+10$ points, some of which are infinitely near to each other. The $10$ additional points are $T_1$, $T_2$ and the $2$ points on each of the exceptional divisors $E_i$ for $i=1,2,3,4$ corresponding to the tangent directions at $Q_i$ specified by respectively $\ell_1$ and $\ell_2$,  $\ell_1$ and $\ell_3$, $m_1$ and $m_3$, and finally $m_2$ and $m_3$.

We will denote by $E_{Q_5}$, $E_{T_1}$ and $E_{T_2}$ the exceptional divisors over $Q_5$, $T_1$ and $T_2$ respectively. These three divisors are $(-1)$-curves.

We will denote by $E_{i}$,
 $F_{i}$, $G_i$ and $H_i$ the four exceptional divisors over $Q_i$ for $i=1,2,3,4$. For each $i$, the divisor $E_i$ is a $(-4)$-divisor intersecting $F_i$, $G_i$, and $H_i$, which are orthogonal $(-1)$-curves.  We make the identification that $F_i$ is the tangent direction corresponding to the basepoint of the pencil of cubics,
 and $H_i$ and $G_i$ correspond to the tangent directions specified in the following table:

\begin{center}
\begin{tabular}{c | c c c c}
$i$ & 1 & 2 & 3 & 4 \\ \hline
$G_i$ & $\ell_1$ & $\ell_3$ & $m_1$ & $m_3$ \\
$H_i$ & $\ell_2$ & $\ell_1$ & $m_3$ & $m_2$.
\end{tabular}
\end{center}
Note that the $E_i$ and $F_j$ are the strict transforms of the curves of the same name on $R_{5,5}$.

The strict transforms of the lines $\ell_j$ and $m_j$, $j=1,2,3$ on $\widetilde{R}$ are the following:
$$\begin{array}{ccc} \ell_1:&=&h-E_{1}-F_1-2G_1-H_1-E_2-F_2-G_2-2H_2-E_{Q_5};\\
 \ell_2:&=&h-E_{1}-F_1-G_1-2H_1-E_4-2F_4-G_4-H_4-E_{T_1};\\
\ell_3:&=&h-E_{2}-F_2-2G_2-H_2-E_3-2F_3-G_3-H_3-E_{T_1};\\
 m_1:&=&h-E_{1}-2F_1-G_1-H_1-E_3-F_3-2G_3-2H_3-E_{T_2};\\
 m_2:&=&h-E_{2}-2F_2-G_2-H_2-E_4-F_4-G_4-2H_4-E_{T_2};\\
 m_3:&=&h-E_{3}-F_3-G_3-2H_3-E_4-F_4-2G_4-H_4-E_{Q_5}.\end{array}$$
The sections of the non--relatively minimal fibration on $\widetilde{R}$ are $F_j$, $j=1,2,3,4$ and $E_{Q_5}$ (i.e. the strict transform of the sections of the fibration on $R_{5,5}$).
\subsubsection{Geometric description of $S_{5,5}$ and its N\'eron--Severi group}

The surface $S_{5,5}$ admits a non-symplectic involution $\iota$ which is the cover involution of the double cover $S_{5,5}\rightarrow \widetilde{R}$. This involution fixes $10$ rational curves (the curves $m_i$, $\ell_i$,$ i=1,2,3$ and $E_j$ with $j=1,2,3,4$) and it acts trivially on the N\'eron-Severi group.

The elliptic fibration $\mathcal{E}_{S_{5,5}}\colon
S_{5,5}\rightarrow \PP^1$ induced by $\mathcal{E}_{R_{5,5}}\colon
R_{5,5}\rightarrow \PP^1$ has two fibers of type $I_{10}$ (induced
by the fibers of type $I_5$ on $R_{5,5}$) and four other singular
fibers, all of type $I_1$. The trivial lattice of the fibration
(generated by the class of the generic fiber, the class of the
zero section and the classes of the non trivial components of the
reducible fibers) has rank 20. The trivial lattice is a sublattice
of the N\'eron--Severi group, and since the N\'eron--Severi group
of a K3 surface has rank at most 20, we conclude that it is
exactly 20. By the Shioda-Tate formula there are no sections of
infinite order for the fibration $\mathcal{E}_{S_{5,5}}\colon
S_{5,5}\rightarrow\PP^1$. The $5$-torsion sections of the
fibration on $R_{5,5}$ induce $5$-torsion sections of
$\mathcal{E}_{S_{5,5}}$. Hence,
$\MW(\mathcal{E}_{S_{5,5}})\supseteq\mathbb{Z}/5\mathbb{Z}$. The
possible torsion parts of the Mordell--Weil group of an elliptic
fibration on a K3 surface are $\Z/n\Z$ for $2\leq n\leq 8$, and
$\Z/2\Z\times \Z/m\Z$, for $m=2,4,6$, $(\Z/k\Z)^2$ for $k=3,4$
(see for example \cite[Thm 7.1]{Shimada}). So we conclude that
$\MW(\mathcal{E}_{S_{5,5}})=\mathbb{Z}/5\mathbb{Z}$.

The curves $\Theta_i^{(j)}$ are in the branch locus and we denote by $\Omega_i^{(j)}$ the rational curve on $S_{5,5}$ which maps $1:1$ to $\Theta_i^{(j)}$. Moreover, we have 10 other rational curves on $S_{5,5}$ : the curves $\Omega_{i_1,i_1-1}^{(j)}$, for $i_1, i_1-1\in \Z/5\Z$, $j=1, 2$, which are the curves resolving the singularities of the intersection point between $\Theta_{i_1}^{(j)}$ and $\Theta_{i_2}^{(j)}$ and are the double cover of the 10 exceptional curves of the blow up $\widetilde{R}\ra R_{5,5}$. The curves $P_i$ are not in the branch locus and we denote by $Q_i$ their $2:1$ cover in $S_{5,5}$. The dual graph of this configuration is given in Figure
\ref{negative_curves_S5511}. We observe that this gives exactly
the diagram given in \cite[Figure 1]{Vin} as dual graph of certain rational curves on the K3 surface whose transcendental lattice is $\langle 2\rangle^2$, which is a different way to describe the surface
$S_{5,5}$.
\begin{figure}\label{negative_curves_S5511}
\begin{tikzpicture}[scale=.7]

\draw (0,4*1) node{$\bullet$};

\draw (4*0.951, 4*0.309) node{$\bullet$};

\draw (4*0.588, 4*-0.809)node{$\bullet$};

\draw (4*-0.588, 4*-0.809)node{$\bullet$};

\draw (4*-0.951, 4*0.309) node{$\bullet$};

\draw (0,3*1) node{$\bullet$};

\draw (3*0.951, 3*0.309) node{$\bullet$};

\draw (3*0.588, 3*-0.809)node{$\bullet$};

\draw (3*-0.588, 3*-0.809)node{$\bullet$};

\draw (3*-0.951, 3*0.309) node{$\bullet$};

\draw (0,2*1) node{$\bullet$};

\draw (2*0.951, 2*0.309) node{$\bullet$};

\draw (2*0.588, 2*-0.809)node{$\bullet$};

\draw (2*-0.588, 2*-0.809)node{$\bullet$};

\draw (2*-0.951, 2*0.309) node{$\bullet$};

{\scriptsize

\draw (0, 0.618) node{$\bullet$};
\draw (0, 0.618) node[above]{$\Omega_{3,2}^{(2)}$};

\draw (0.618*0.951, 0.618*0.309) node{$\bullet$};
\draw (0.618*0.951, 0.618*0.309+.18) node[right]{$\Omega_{4,0}^{(2)}$};

\draw (0.618*0.588, 0.618*-0.809)node{$\bullet$};
\draw (0.618*0.588-.23, 0.618*-0.809+.05)node[below right]{$\Omega_{2,1}^{(2)}$};

\draw (0.618*-0.588, 0.618*-0.809)node{$\bullet$};
\draw (0.618*-0.588+.16, 0.618*-0.809+.05)node[below left]{$\Omega_{4,3}^{(2)}$};

\draw (0.618*-0.951, 0.618*0.309) node{$\bullet$};
\draw (0.618*-0.951, 0.618*0.309+.18) node[left]{$\Omega_{1,0}^{(2)}$};

\draw (0, 2*-1.618) node{$\bullet$};
\draw (0, 2*-1.618) node[below]{$\Omega_{3,2}^{(1)}$};

\draw (2*-1.618*0.951, 2*-1.618*0.309) node{$\bullet$};
\draw (2*-1.618*0.951, 2*-1.618*0.309) node[below left]{$\Omega_{4,3}^{(1)}$};

\draw (2*-1.618*0.588, 2*-1.618*-0.809)node{$\bullet$};
\draw (2*-1.618*0.588, 2*-1.618*-0.809) node[above left]{$\Omega_{4,0}^{(1)}$};

\draw (2*-1.618*-0.588, 2*-1.618*-0.809)node{$\bullet$};
\draw (2*-1.618*-0.588, 2*-1.618*-0.809) node[above right]{$\Omega_{1,0}^{(1)}$};

\draw (2*-1.618*-0.951, 2*-1.618*0.309) node{$\bullet$};
\draw (2*-1.618*-0.951, 2*-1.618*0.309) node[below right]{$\Omega_{2,1}^{(1)}$};

}
\draw (0,4*1) --  (4*0.951, 4*0.309) --(4*0.588, 4*-0.809) -- (4*-0.588, 4*-0.809) -- (4*-0.951, 4*0.309) --(0,4);

\draw (0,2*1)--(2*0.588, 2*-0.809) -- (2*-0.951, 2*0.309)  --  (2*0.951, 2*0.309) -- (2*-0.588, 2*-0.809) --(0,2);

\draw (0,4)node[above]{$\Omega_0^{(1)}$} -- (0,3)node[left]{$Q_0$} -- (0,2)node[right]{$\Omega_0^{(2)}$};

\draw (4*0.951, 4*0.309)node[above right]{$\Omega_1^{(1)}$} -- (3*0.951, 3*0.309)node[below]{$Q_1$}-- (2*0.951, 2*0.309)node[above]{$\Omega_2^{(2)}$};

\draw (4*0.588, 4*-0.809)node[below]{$\Omega_2^{(1)}$} -- (3*0.588, 3*-0.809)node[left]{$Q_2$} --(2*0.588, 2*-0.809)node[right]{$\Omega_4^{(2)}$};

\draw (4*-0.588, 4*-0.809)node[below]{$\Omega_3^{(1)}$} -- (3*-0.588, 3*-0.809)node[right]{$Q_3$} -- (2*-0.588, 2*-0.809)node[left]{$\Omega_1^{(2)}$};

\draw (4*-0.951, 4*0.309)node[above left]{$\Omega_4^{(1)}$} --(3*-0.951, 3*0.309)node[below]{$Q_4$} -- (2*-0.951, 2*0.309)node[above]{$\Omega_3^{(2)}$};

\end{tikzpicture}
\caption{Dual graph of relevant negative curves on $S_{5,5}$.}\label{negative_curves_S5511}

\end{figure}
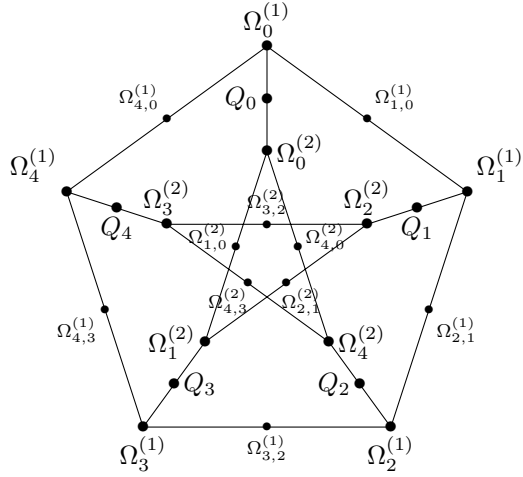

Let $\pi \colon S_{5,5} \to \widetilde{R}$ denote the double cover.  As can be deduced from \eqref{R_5,5 curve ids} and the above identifications, pushing forward curve classes has the following effect:

\begin{align}\label{tildeR curve ids}
\begin{array}{c c c c c}
\pi_*\Omega_0^{(1)} = m_1 &\pi_*\Omega_{3,2}^{(1)}=2G_4 & \pi_*\Omega_0^{(2)} = E_1 &\pi_*\Omega_{3,2}^{(2)}=2G_2& \pi_*Q_0 = 2F_1 \\
\pi_*\Omega_{1,0}^{(1)}=2G_3 & \pi_*\Omega_3^{(1)} = E_4 & \pi_*\Omega_{1,0}^{(2)}=2H_1 &\pi_*\Omega_3^{(2)} = E_2 & \pi_*Q_3 = 2F_4\\
\pi_*\Omega_1^{(1)} = E_3 & \pi_*\Omega_{4,3}^{(1)}=2H_4 & \pi_*\Omega_1^{(2)} = \ell_2 & \pi_*\Omega_{4,3}^{(2)}=2H_2 & \pi_*Q_1 = 2F_3\\
\pi_*\Omega_{2,1}^{(1)}=2H_3 &\pi_*\Omega_4^{(1)} = m_2 & \pi_*\Omega_{1,0}^{(2)}=2E_{T_1} &  \pi_*\Omega_4^{(2)} = \ell_1 & \pi_*Q_4 = 2F_2 \\
\pi_*\Omega_2^{(1)} = m_3 & \pi_*\Omega_{4,0}^{(1)}=2E_{T_2} & \pi_*\Omega_2^{(2)} = \ell_3 & \pi_*\Omega_{4,0}^{(2)}=2G_1& \pi_*Q_2 = 2E_5.
\end{array}
\end{align}



\subsubsection{Weierstrass equation of $S_{5,5}$}\label{subsec: Weierstrass S55}
By the Weierstrass equation \eqref{eq: Weierstrass R5511} the
fibers of $\E_{R_{5,5}}$ of type $I_5$  are the fibers over $\mu=0$ and $\mu=\infty$.
So the base change branched on these fibers is given by $\mu\ra
\mu^2$ and the elliptic fibration on $S_{5,5}$ induced by the one on
$R_{5,5}$ is
\begin{equation}\label{eq: Weierstrass S, induced by R}
y^2=x^3+A(\mu)x+B(\mu), \ \ \mbox{where}\end{equation}
$$A(\mu):=\frac{-1}{48}\mu^8-\frac{1}{4}\mu^6-\frac{7}{24}\mu^4+\frac{1}{4}\mu^2-\frac{1}{48},\mbox{  and }$$
$$B(\mu):=-\frac{1}{864}\mu^{12}-\frac{1}{48}\mu^{10}-\frac{25}{288}\mu^8-\frac{25}{288}\mu^4+\frac{1}{48}\mu^2-\frac{1}{864}. $$
The discriminant is $\frac{1}{16}\mu^{10}(-1+11\mu^2+\mu^4)$, so there are, as expected, two fibers of type $I_{10}$ over $\mu=0$ and $\mu=\infty$. Moreover there are four fibers of type $I_1$ over $\mu=\pm\sqrt{-\frac{11}{2}\pm\frac{5}{2}\sqrt{5}}$.

\subsubsection{Double cover of $\mathbb{P}^2$}\label{subsec: S55 double cover P2}

On the other hand, $\widetilde{R}$ and $R_{5,5}$ are blow-ups of $\mathbb{P}^2$ and
the branch fibers of $\pi:S_{5,5}\dashrightarrow R_{5,5}$ corresponds to the cubics $f_3:=x_1x_2(x_0-x_1) =0$ and
$g_3:=x_0(x_0-x_1-x_2)(x_0-x_2) =0$. This exhibits $S_{5,5}$ as a double cover of
$\mathbb{P}^2$ branched along the union of these two cubics. So we
obtain a different equation for $S_{5,5}$, as a double cover of
$\mathbb{P}^2$, i.e.
\begin{equation}\label{eq: S 2:1 cover of P2}w^2=x_1x_2(x_0-x_1)x_0(x_0-x_1-x_2)(x_0-x_2).\end{equation}
We observe that $S_{5,5}$ is rigid (both in the moduli space of the elliptic K3 surfaces with prescribed reducible fibers and in the moduli space of the K3 surfaces with a non--symplectic involution with a prescribed fixed locus), since both $R_{5,5}$ and the
choice of the branch fibers are.

\subsection{The branch fibers are $2I_0$: the K3 surface
$X_{5,5}$}\label{subsec: K3 double cover branch 2I0}

Let us consider the K3 surface $X_{5,5}$ obtained by a base change
of order 2 of $R_{5,5}$ whose branch locus corresponds to two
fibers of type $I_0$. Let us assume it is very general among the
K3 obtained in this way. This K3 surface lies in a 2-dimensional
family of K3 surfaces (see \cite{GSarti}), whose parameters depend
on the choice of the two branch fibers (see \eqref{eq: Weierstrass
X55, induced by R}).

\subsubsection{Geometric description of $X_{5,5}$ and its N\'eron--Severi group}

The surface $X_{5,5}$ admits a non-symplectic involution $\iota$ which is the cover involution of the double cover $X_{5,5}\rightarrow R_{5,5}$ and which fixes two elliptic curves.

The elliptic fibration $\mathcal{E}_{X_{5,5}}: X_{5,5}\rightarrow \PP^1$ induced by $\mathcal{E}_{R_{5,5}}:R_{5,5} \rightarrow \PP^1$ has four fibers of type $I_5$ and four fibers of type $I_1$. Moreover it has a 5-torsion section, induced by the one of sections of the elliptic fibration on $\mathcal{E}_{R_{5,5}}$. The N\'eron--Severi group and the transcendental lattice of this K3 surface are computed in \cite{GSarti} and a set of generators of the N\'eron--Severi group is given by the class of the fiber of the fibration, the zero section, one section of order 5 and the irreducible components of the reducible fibers of the fibration.

The curves $\Theta_i^{(j)}$ are not in the branch locus and we denote by $\Omega_i^{(j,k)}$ for $k=1,2$ the two disjoint rational curves which are mapped to $\Theta_i^{(j)}$ by the quotient map $X_{5,5}\ra R_{5,5}$. The curves $P_i$ are not in the branch locus and we denote by $Q_i$ their $2:1$ cover in $X_{5,5}$.  The dual graph of this configuration is shown in Figure \ref{negative_curves_X55}. This diagram is also a tropical surface, similar to the one given in \cite[Figure 1]{RSS}, as pointed out by B.\ Sturmfels.


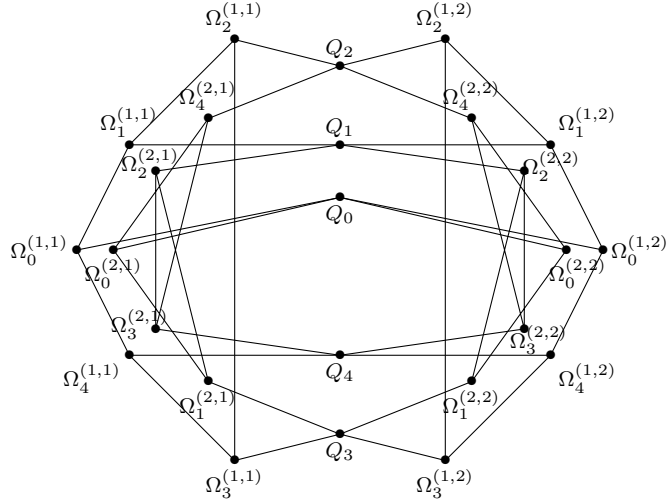
\begin{figure}[h]
\begin{tikzpicture}[scale=.7]
{\footnotesize

\draw (-2,4) node{$\bullet$};

\draw (-2,-4) node{$\bullet$};

\draw (-4,2) node{$\bullet$};

\draw (-4,-2) node{$\bullet$};

\draw (-5,0) node{$\bullet$};

\draw (2,4) node{$\bullet$};

\draw (2,-4) node{$\bullet$};

\draw (4,2) node{$\bullet$};

\draw (4,-2) node{$\bullet$};

\draw (5,0) node{$\bullet$};

\draw (0,3.5) node{$\bullet$};
\draw (0,3.5) node[above]{$Q_2$};

\draw (0,2) node{$\bullet$};
\draw (0,2) node[above]{$Q_1$};

\draw (0,1) node{$\bullet$};
\draw (0,1) node[below]{$Q_0$};

\draw (0,-2) node{$\bullet$};
\draw (0,-2) node[below]{$Q_4$};

\draw (0,-3.5) node{$\bullet$};
\draw (0,-3.5) node[below]{$Q_3$};

\draw (-2,4) -- (-2,-4)--(-4,-2)--(-5,0)--(-4,2)--(-2,4);

\draw (2,4) -- (2,-4)--(4,-2)--(5,0)--(4,2)--(2,4);

\draw (4.3,0) node{$\bullet$};

\draw (3.5,1.5) node{$\bullet$};

\draw (3.5,-1.5) node{$\bullet$};

\draw (2.5,2.5) node{$\bullet$};

\draw (2.5,-2.5) node{$\bullet$};

\draw (4.3,0) -- (2.5,-2.5) -- (3.5,1.5) -- (3.5,-1.5) -- (2.5,2.5) -- (4.3,0);

\draw (-4.3,0) node{$\bullet$};

\draw (-3.5,1.5) node{$\bullet$};

\draw (-3.5,-1.5) node{$\bullet$};

\draw (-2.5,2.5) node{$\bullet$};

\draw (-2.5,-2.5) node{$\bullet$};

\draw (-4.3,0) -- (-2.5,-2.5) -- (-3.5,1.5) -- (-3.5,-1.5) -- (-2.5,2.5) -- (-4.3,0);

\draw (-2,4)node[above]{$\Omega_2^{(1,1)}$} -- (0,3.5);
\draw (-2.5,2.5)node[above]{$\Omega_4^{(2,1)}$} -- (0,3.5);

\draw (2,4)node[above]{$\Omega_2^{(1,2)}$} -- (0,3.5);
\draw (2.5,2.5)node[above]{$\Omega_4^{(2,2)}$} -- (0,3.5);

\draw (-4,2)node[above]{$\Omega_1^{(1,1)}$} -- (0,2);
\draw (-3.5,1.5) -- (0,2);

\draw (-3.5-.1,1.5+.15) node{$\Omega_2^{(2,1)}$} ;

\draw (4,2)node[above right]{$\Omega_1^{(1,2)}$}  -- (0,2);
\draw (3.5,1.5) -- (0,2);

\draw (3.5+.52,1.5+.1) node{$\Omega_2^{(2,2)}$} ;

\draw (-5,0)node[left]{$\Omega_0^{(1,1)}$} -- (0,1);
\draw (-4.3,0)node[below]{$\Omega_0^{(2,1)}$} -- (0,1);

\draw (5,0)node[right]{$\Omega_0^{(1,2)}$}  -- (0,1);
\draw (4.3,0) -- (0,1);

\draw (4.3+.22,0)node[below]{$\Omega_0^{(2,2)}$};

\draw (-4,-2)node[below left]{$\Omega_4^{(1,1)}$} -- (0,-2);
\draw (-3.5,-1.5) -- (0,-2);

\draw (-3.5-.3,-1.5+.1) node{$\Omega_3^{(2,1)}$} ;

\draw (4,-2)node[below right]{$\Omega_4^{(1,2)}$} -- (0,-2);
\draw (3.5,-1.5) -- (0,-2);

\draw (3.5+.28,-1.5 -.2)node{$\Omega_3^{(2,2)}$};

\draw (-2,-4)node[below]{$\Omega_3^{(1,1)}$} -- (0,-3.5);
\draw (-2.5,-2.5)node[below]{$\Omega_1^{(2,1)}$} -- (0,-3.5);

\draw (2,-4)node[below]{$\Omega_3^{(1,2)}$} -- (0,-3.5);
\draw (2.5,-2.5)node[below]{$\Omega_1^{(2,2)}$} -- (0,-3.5);
}
\end{tikzpicture}
\caption{Dual graph of relevant negative curves on $X_{5,5}$.}\label{negative_curves_X55}
\end{figure}

\subsubsection{Weierstrass equation of $X_{5,5}$}\label{sub: Weierstrass X55}
Let us denote by $\mu_1$ and $\mu_2$ two arbitrary points of
$\mathbb{P}^1_{\mu}$ such that the fibers of \eqref{eq:
Weierstrass R5511} over $\mu_1$ and $\mu_2$ are smooth. Let
$X_{5,5}$ be the surface obtained from $R_{5,5}$ by a base change
of order 2 branched in $\mu_1$ and $\mu_2$. We already observed
that the surface $X_{5,5}$ lives in a 2-dimensional family of K3
surfaces and its equation depends on the two parameters $\mu_1$
and $\mu_2$.

Let us consider the base change $\mathbb{P}^1_{(\alpha:\beta)}\ra\mathbb{P}^1_{(\mu:\lambda)}$ branched over $(\mu:\lambda)=(\mu_1:1)$ and $(\mu_2:1)$, i.e. the base change given by
\begin{equation}\label{eq: base change}\mu=\mu_1 \alpha^2+\beta^2,\ \ \ \ \ \lambda=\alpha^2+\beta^2/\mu_2.\end{equation} It induces on $X_{5,5}$ the elliptic fibration
\begin{equation}\label{eq: Weierstrass X55, induced by R}
y^2=x^3+A(\alpha:\beta)x+B(\alpha:\beta)
\end{equation}
whose discriminant is
\begin{align}\label{eq:disc generic base change}\begin{array}{l}
\left(\left(\alpha^2\mu_2+\beta^2\right)^5\left(\mu_1\alpha^2+\beta^2\right)^5\left(\alpha^4\mu_1^2\mu_2^2+11\alpha^4\mu_1\mu_2^2-\alpha^4\mu_2^2+11\alpha^2\mu_1\mu_2\beta^2+\right.\right.\\
\left.\left.+2\mu_1\alpha^2\mu_2^2\beta^2+11\alpha^2\mu_2^2\beta^2-2\alpha^2\mu_2\beta^2+\mu_2^2\beta^4+11\beta^4\mu_2-\beta^4\right)\right)/\left(16\mu_2^7\right)\end{array}\end{align}

For generic values of $\mu_1$ and $\mu_2$ the elliptic fibration has $4I_5+4I_1$ as singular fibers.
\subsubsection{Double cover of $\mathbb{P}^2$}\label{sub: double cover P2 X55}
On the other hand, $X_{5,5}$ is the double cover of
$\mathbb{P}^2$ branched on the union of the two cubics
$x_1x_2(x_0-x_1)+\mu_1x_0(x_0-x_1-x_2)(x_0-x_2) =0$ and
$x_1x_2(x_0-x_1)+\mu_2x_0(x_0-x_1-x_2)(x_0-x_2) =0$. So $X_{5,5}$ can
be described by the equation
\begin{multline}\label{eq: X55 2:1 cover of P2}w^2=\bigl(x_1x_2(x_0-x_1)+\mu_1x_0(x_0-x_1-x_2)(x_0-x_2)\bigr)\\ \bigl(x_1x_2(x_0-x_1)+\mu_2x_0(x_0-x_1-x_2)(x_0-x_2)\bigr).\end{multline}

\subsection{Branch fibers $I_5$ and $I_1$}\label{subsec: double cover I5 I1}
If one uses as branch fibers a fiber of type $I_5$  and one of
type $I_1$ one obtains a rigid K3 surface (in the moduli space of the elliptic K3 surfaces with prescribed reducible fibers), whose singular fibers
are $I_{10}+2I_5+I_2+2I_1$ and theoretically one has four
different admissible choices to do that. Indeed one can choose the fiber of type $I_5$ which is the branch fiber to be the
fiber either over $\mu_1=0$ or over $\mu_1=\infty$ and similarly
one can choose the fiber of type $I_1$ which is the branch
fiber to be the fiber either over $\mu_2=(-11+5\sqrt{5})/2$ or over
$\mu_2=(-11-5\sqrt{5})/2$. In order to obtain the Weierstrass
equation of these elliptic fibrations it suffices to apply the
base change \eqref{eq: base change} with the chosen values for
$\mu_1$ and $\mu_2$.

We assume that $\mu_2=(-11+5\sqrt{5})/2$ and we obtain the following two Weierstrass equations. If $\mu_1=0$ the base change \eqref{eq: base change} applied to $\mu_1=0$ and $\mu_2=(-11+5\sqrt{5})/2$ gives an elliptic fibration:

\begin{equation*}
y^2=x^3+A(\alpha:\beta)x+B(\alpha:\beta)
\end{equation*}
whose discriminant is
\begin{equation*}
-\frac{1}{512}\left(-\alpha^2+5\beta^2\sqrt{5}\right)\left(-2\alpha^2-11\beta^2+5\beta^2\sqrt{5}\right)^5\alpha^2\beta^{10}.
\end{equation*}

In order to choose $\mu_1=\infty$, one has to slightly change the equation of the base change \eqref{eq: base change}, which now is $\mu=\alpha^2$ and $\lambda=\alpha^2/\mu_2+\beta^2$ and one obtains

\begin{equation*}y^2=x^3+A(\alpha:\beta)x+B(\alpha:\beta)
\end{equation*}
whose discriminant is
\begin{equation*}
\frac{1}{5120}\left(-375125+167761\sqrt{5}\right)\left(-25\alpha^2+\beta^2\sqrt{5}\right)\left(-2\alpha^2+11\beta^2+5\beta^2\sqrt{5}\right)^5\beta^2\alpha^{10}.\end{equation*}
We observe that in the first case the K3 surface obtained is described as a double cover of $\mathbb{P}^2$ by the equation
$$w^2=x_1 x_2 (x_0-x_1)\left(x_0x_1x_2(13+5\sqrt{5})-x_1^2 x_2(11+5\sqrt{5})+2x_0^3-4x_0^2x_2-2x_1x_0^2+2x_0x_2^2\right),$$
in the second by the equation
$$w^2= x_0(x_0-x_1-x_2)(x_0-x_2)\left(x_0x_1x_2(13+5\sqrt{5})-x_1^2x_2(11+5\sqrt{5})+2x_0^3-4x_0^2x_2-2x_1x_0^2+2x_0x_2^2\right).$$

\subsection{Branch fibers $I_5$ and $I_0$}
If one chooses as branch fibers one fiber of type $I_5$ and one of type $I_0$ one obtains a 1-dimensional family of K3 surfaces, whose singular fibers are $I_{10}+2I_5+4I_1$ and theoretically one has two different admissible ways to do that.  Indeed one can chose that the fiber of type $I_5$ which is the branch fiber is the fiber either over $\mu_1=0$ or over $\mu_1=\infty$, while $\mu_2$ is the parameter of the family. In order to obtain the equations of these elliptic fibrations one has to apply the base changes $(\mu=\beta^2,\lambda=\alpha^2+\beta^2/\mu_2)$ or $(\mu=\alpha^2,\lambda=\alpha^2/\mu_2+\beta^2)$ to the equation \eqref{eq: Weierstrass R5511}, exactly as in the previous sections.

Similarly one can describe these K3 surfaces as a double cover of
$\mathbb{P}^2$ substituting in \eqref{eq: X55 2:1 cover of
P2} the appropriate values of $\mu_1$ and $\mu_2$.
\subsection{Branch fibers $I_1$ and $I_0$}
If one chooses as branch fibers one fiber of type $I_1$  and one of type $I_0$ one obtains a 1-dimensional family of K3 surfaces, whose singular fibers are $4I_5+I_2+2I_1$  and theoretically one has two different admissible ways to do that.  Indeed one can choose that the fiber of type $I_1$ which is the branch fiber is the fiber either over $\mu_1=(-11-5\sqrt{5})/2$ or over $\mu_1=(-11+5\sqrt{5})/2$, while $\mu_2$ is the parameter of the family. In order to obtain the equations of these elliptic fibrations one has to apply the base change \eqref{eq: base change} with the chosen $\mu_1$ to the equation \eqref{eq: Weierstrass R5511} and to obtain an equation of this surface as a double cover of $\mathbb{P}^2$
one has to substitute the chosen $\mu_1$ in \eqref{eq: X55 2:1 cover of P2}.
\subsection{Branch fibers $2I_1$}
If one chooses as branch fibers the two fibers of type $I_1$ one
obtains a rigid K3 surface, whose reducible fibers are $4I_5+2I_2$.  In order
to obtain the equation of this elliptic fibration one has to apply
the base change \eqref{eq: base change} with the chosen
$\mu_1=(-11-5\sqrt{5})/2$ and $\mu_2=(-11+5\sqrt{5})/2$ to the
equation \eqref{eq: Weierstrass R5511}. Similarly to obtain an
equation of this surface as a double cover of $\mathbb{P}^2$ one
has to substitute  $\mu_1=(-11-5\sqrt{5})/2$ and
$\mu_2=(-11+5\sqrt{5})/2$ in \eqref{eq: X55 2:1 cover of P2}.

\section{Elliptic fibrations on K3 surfaces induced by the conic
bundles}\label{sec: elliptic fibrations induced by conic bundles}
The aim of this section is to describe both geometrically and by
the Weierstrass equations the elliptic fibrations induced by the
conic bundles $B_i$ (described in Section \ref{sec: conic
bundles}) on the K3 surfaces described in Section \ref{sec: K3
covers of  R55}. We also provided a general method to find these
Weierstrass equations, under some assumption on the conic bundles,
see Section \ref{subsec: an algorithm (generalized)  conic
bundle}.

\subsection{An example}
Let us consider the K3 surface $S_{5,5}$, whose equation as a double
cover of $\mathbb{P}^2$ is given by \eqref{eq: S 2:1 cover of P2}. Let us consider also the conic bundle $|B_1|$ from Section \ref{sec:
conic bundles}. By \cite[Theorem 5.3]{GS}, the conic bundle $\vert B_1\vert$ induces an
elliptic fibration on $S_{5,5}$ with three reducible fibers of
type $I_2^*$: one whose components are $Q_0$,
$\Omega_{4,0}^{(1)}$, $\Omega_0^{(1)}$, $\Omega_{1,0}^{(1)}$,
$\Omega_1^{(1)}$, $Q_1$, and $\Omega_{2,1}^{(1)}$, one whose
components are $Q_2$, $\Omega_{4,0}^{(2)}$, $\Omega_4^{(2)}$,
$\Omega_{4,3}^{(2)}$, $\Omega_3^{(2)}$, $Q_4$, and
$\Omega_{3,2}^{(2)}$, and one whose components are
$\Omega_{4,3}^{(1)}$, $\Omega_{3,2}^{(1)}$, $\Omega_3^{(1)}$,
$Q_3$, $\Omega_{1}^{(2)}$, $\Omega_{2,1}^{(2)}$, and
$\Omega_{1,0}^{(2)}$.

\subsubsection{Equation of the elliptic fibration on $S_{5,5}$ induced by $|B_1|$}
Let us consider the conic bundle $B_1$ on $R_{5,5}$ associated to
the pencil of lines $x_1=\tau x_0\subset \mathbb{P}^2$. It induces
an elliptic fibration on $S_{5,5}$. To find the equation of this
elliptic fibration we use the equation of $S_{5,5}$ as double cover of
$\mathbb{P}^2$, i.e. the equation \eqref{eq: S 2:1 cover of  P2} and we
substitute in $x_1=\tau x_0$ in \eqref{eq: S 2:1 cover of  P2}.

This gives:
$$w^2= (\tau x_0)x_2(x_0-\tau x_0)x_0(x_0-\tau x_0-x_2)(x_0-x_2).$$
We put $x_2=1$ and we obtain
$$w^2= \tau (1-\tau)x_0^3(x_0-\tau x_0-1)(x_0-1).$$

Let us consider the change of coordinates $w\mapsto wx_0$ and divide both the members by $x_0^2$. We
obtain
$$w^2= \tau (1-\tau)x_0(x_0-\tau x_0-1)(x_0-1).$$

This is the equation of an elliptic fibration over $\mathbb{P}^1_\tau$. 
Moreover one can explicitly compute the Weierstrass
form: first one uses the change of coordinates $w\mapsto
\tau^2(1-\tau)w$ and $x_0\mapsto \tau x_0$ obtaining
$$
w^2\tau^4(1-\tau)^2=\tau^2(1-\tau)x_0(\tau x_0(1-\tau)-1)(\tau
x_0-x_2)
$$   and so
$$
w^2=x_0\left(x_0-\frac{1}{\tau(1-\tau)}\right)\left(x_0-\frac{1}{\tau}\right)
$$
Second, one considers the change of coordinates $w\mapsto
w/\tau^3(1-\tau)^3$ and $x_0\mapsto x_0/\tau^2(1-\tau)^2$ and
multiplies all the equation by $\tau^6(1-\tau)^6$. So one obtains
\begin{eqnarray}\label{eq: w B1 on S}w^2=x_0(x_0-\tau(1-\tau))(x_0-\tau(1-\tau)^2).\end{eqnarray}
The discriminant is $\tau^8(1-\tau)^8$ and so, by Tate's algorithm, there are three
fibers of type $I_2^*$ over $\tau=0$, $\tau=1$, $\tau=\infty$.

\subsection{An algorithm to compute Weierstrass equations}\label{subsec: an algorithm (generalized) conic bundle}

The aim of this section is to formalize systematically what we did
above.

\noindent\textbf{Setup.} Let $V$ be a K3 surface obtained by a
base change of order 2 from a rational elliptic surface $R$.
Therefore, $V$ can be described as double cover of $\mathbb{P}^2$
branched on the union of two (possibly reducible) plane cubics
from the pencil determining $R$. It has an equation of the form
\begin{equation}\label{eq: V double cover P2}w^2=f_3(x_0:x_1:x_2)g_3(x_0:x_1:x_2).\end{equation}
Let $B$ be a conic bundle on $R$, e.g. a basepoint-free linear
system of rational curves giving $R \rightarrow\mathbb{P}^1_\tau$.
Pushing forward to $\mathbb{P}^2$, $B$ is given by a pencil of
plane rational curves with equation $h(x_0:x_1:x_2,\tau)$.  The
polynomial $h(x_0:x_1:x_2,\tau)$ is homogeneous in $x_0$, $x_1$,
$x_2$, say of degree $e\geq 1$ and linear in $\tau$.

As the anticanonical series on $R$ coincides with the elliptic
fibration, the adjunction formula implies that every curve with
equation $h(x_0:x_1:x_2,\tau)$ meets both of the branch curves
(the proper transforms on $R$ of) $f_3=0$ and $g_3=0$ in two
additional points.  It therefore meets (the proper transform of)
their union $f_3g_3=0$ in four points.  (Note that there may be
additional points of intersection on $\mathbb{P}^2$ which are
separated in the blow-up $R$).  Therefore the preimage in $V$ is
the double cover of a rational curve branched over $4$ points,
e.g. the standard presentation of an elliptic curve.  For general
$\tau$, we must find an isomorphism of the curve
$h(x_0:x_1:x_2,\tau) =0$ with $\mathbb{P}^1$, and extract the
images of the four intersection points with $f_3g_3 =0$.

When $e \leq 3$, an isomorphism with $\mathbb{P}^1$ is provided by
projection from a point of order $e-1$ on the curve (e.g. any
point in $\mathbb{P}^2$ if $e=1$, a point on the conic if $e=2$,
and a double point of the cubic if $e=3$).  Such a point
necessarily exists (in the case $e=3$ the singularity must be a
basepoint of the pencil) and is also necessarily a basepoint of
the original pencil of cubics giving $\E_R$.  Up to acting by
$\PGL_3(\C)$, we may assume that this point is $(0:1:0)$.


\noindent\textbf{Algorithm when $e\leq 3$.}
\begin{enumerate}
\item Compute the resultant of the polynomials
$f_3(x_0:x_1:x_2)g_3(x_0:x_1:x_2)$ and $h(x_0:x_1:x_2,\tau)$ with
respect to the variable $x_1$. The result is a polynomial
$r(x_0:x_2,\tau)$ which is homogeneous in $x_0$ and $x_2$,
corresponding to the images of \textit{all} of the intersection
points $\{f_3g_3=0 \} \cap \{h_\tau=0\}$ after projection from
$(0:1:0)$.

\item Since $B$ is a conic bundle, $r(x_0:x_2,\tau)$ will be of
the form $a(x_0:x_2,\tau)^2b(x_0:x_2,\tau)c(\tau)$, where $a$ and
$b$ are homogeneous in $x_0$ and $x_2$, the degree of $a$ depends
upon $e$ and the degree of $b$ in $x_0$ and $x_2$ is 4. \item The
equation of $V$ is now given by $w^2=r(x_0:x_2,\tau)$, which is
birationally equivalent to
\begin{equation}\label{eq: genus 1 fibration by conic bundles}w^2=c(\tau)b(x_0:x_2,\tau),\end{equation} by the change of coordinates
$w\mapsto wa(x_0:x_2,\tau)$. Since for almost every $\tau$, the
equation \eqref{eq: genus 1 fibration by conic bundles} is the
equation of a $2:1$ cover of $\mathbb{P}^1_{(x_0:x_2)}$ branched
in 4 points, \eqref{eq: genus 1 fibration by conic bundles} is the
equation of the genus 1 fibration on the K3 surface $V$ induced by
the conic bundle $B$. \item If there is a section of fibration
\eqref{eq: genus 1 fibration by conic bundles}, then it is
possible to obtain the Weierstrass form by standard
transformations.\end{enumerate}

\begin{rem}{\rm The algorithm can be applied exactly in the same way to the generalized conic bundles, and not only to the conic bundles.}\end{rem}

When $e\geq 4$, projection from a point may suffice, for example
if all curves have a basepoint of degree $e-1$. However there are
several conic bundles whose general member can not be parametrized
by lines.

We now consider a parametrization which can be done by conics.
Let $\mathcal{P}$ be a pencil of rational curves of degree $e$
passing through the (possibly infinitely near) basepoints
$T_1,\ldots T_{r}$ with certain multiplicities. Let $\mathcal{C}$
be a pencil of conics whose basepoints are among $T_1,\ldots
T_{r}$ and such that $2e-1$ intersection points between a generic
curve in $\mathcal{P}$ and a generic curve in $\mathcal{C}$ are in
$\{T_1,\ldots T_r\}$. So there is exactly one extra intersection
between a generic curve in $\mathcal{P}$ and a generic curve in
$\mathcal{C}$. This allows to parametrize the curves in
$\mathcal{P}$ by the curves in $\mathcal{C}$. If moreover the base
points of $\mathcal{C}$ are three distinct points and one
infinitely near point we will say that the condition $(\star)$ is
satisfied. So $(\star)$ is satisfied if the curves in
$\mathcal{P}$ can be parametrized by a pencil of conics passing
through 4 points, exactly two of which are infinitely near.

If the condition $(\star)$ is satisfied, up to changing
coordinates by some matrix $M \in \PGL_3(\cc)$, we may assume that
the basepoints of $\mathcal{C}$ are  $p_1 = (0:0:1), p_2 =
(0:1:0), p_3 = (1:0:0)$ and the infinitely near point $p_1'$
corresponds to the line $x_0=x_1$.  The pencil of degree $2$ maps
$\pp^1_z \to \pp^2$ sending
\[1 \mapsto p_1, \qquad \infty \mapsto p_2, \qquad 0 \mapsto p_3 \]
with derivative at $z=1$ in the direction of the line $x_0=x_1$ is
given by
\[(x_0:x_1:x_2) = (z-1:z(z-1):p\cdot z), \qquad p \in \pp^1. \]

In the following we are interested in pencils of quartics which
satisfy the condition $(\star)$, so we study the effect of this condition on quartic curves. We say that a pencil of quartics satisfies
$(\dagger)$ if, up to a change of coordinates, it is of one of the
following types: (1) each quartic is double at $p_2$ and has a
tacnode at $p_1$ with principal tangent specified by $p_1'$; (2)
each quartic is double at $p_2$ and $p_3$ and has a cusp at $p_1$
with principal tangent specified by $p_1'$.

We recall that, by the construction of the conic bundles, all the
basepoints of $\mathcal{P}$ (and thus also of $\mathcal{C}$) are
also singular points for the sextic $f_3\cdot g_3=0$.

\noindent\textbf{Algorithm assuming $(\dagger)$.}
\begin{enumerate}
\item Factor $h(z-1:z(z-1):pz,\tau) = c(\tau)z^a(z-1)^br(z, p, \tau)$ where $r(z, p, \tau)$ is linear in $z$ and $a+b$ is $5$ or $6$ depending on the multiplicity of $h$ at $p_2$.  The solution $z_0$ of $r(z, p, \tau)=0$ gives the rational parameterization $h(z_0-1:z_0(z_0-1):pz_0,\tau)=0$ with parameter $p \in \pp^1$.

\item  A birational equation of the K3 surface is given by
\[w^2 = (f_3 \cdot g_3)(z_0-1:z_0(z_0-1):pz_0). \]

\item If there is a section of fibration \eqref{eq: genus 1
fibration by conic bundles}, then it is possible to obtain the
Weierstrass form by standard transformations.\end{enumerate} This
algorithm can be generalized to pencil of curves satisfying
$(\star)$.

\subsection{The elliptic fibrations induced by conic bundles}\label{subsec: elliptic fibrations induced by conic bundle}
Here we can describe and compute the equations of all the elliptic
fibrations induced by the conic bundles on the different K3
surfaces introduced. For this purpose, we apply the algorithm, described in the previous section, to the equations of the conic bundles given in Section \ref{sec:
conic bundles} and to the Weierstrass equations given in Section \ref{sec:
K3 covers of R55}.

\subsubsection{The K3 surface $S_{5,5}$}\label{subsec: fibration induced by conic bundles S}
The conic bundle $\vert B_2\vert$ induces an elliptic fibration on
$S_{5,5}$ with two reducible fibers of type $I_4^*$: one whose
components are $Q_0$, $\Omega_{0,4}^{(1)}$, $\Omega_0^{(1)}$,
$\Omega_{1,0}^{(1)}$, $\Omega_1^{(1)}$, $\Omega_{2,1}^{(1)}$,
$\Omega_2^{(1)}$, $Q_2$, and $\Omega_{3,2}^{(1)}$, and one whose
components are $Q_3$, $\Omega_{1,0}^{(2)}$, $\Omega_1^{(2)}$,
$\Omega_{2,1}^{(2)}$, $\Omega_2^{(2)}$, $\Omega_{3,2}^{(2)}$,
$\Omega_3^{(2)}$, $Q_4$, and $\Omega_{4,3}^{(2)}$.

The Weierstrass equation is computed applying the algorithm to $f_3g_3=x_0x_1x_2(x_0-x_1)(x_0-x_2)(x_0-x_1-x_2)$, by \eqref{eq: S 2:1 cover of  P2}, and $h(x_0:x_1:x_2,\tau)=x_1x_2-\tau(x_0x_2-x_0^2)$. One obtains the Weierstrass equation \begin{equation}\label{eq: w B2 on S} y^2=x^3-\frac{\tau^2(-\tau^2+\tau^4+1)}{3}x-(1/27)\frac{\tau^3(-2+\tau^2)(-1+2\tau^2)(\tau^2+1)}{27}.\end{equation}  So the discriminant $\Delta(\tau)$ is $-\tau^{10}(-1+\tau)^2(\tau+1)^2$. Hence, by Tate's algorithm, this fibration has two fibers of type $I_4^*$ over $\tau=0,\infty$ and two fibers of type $I_2$ over $\tau=\pm 1$.

The conic bundle $\vert B_3\vert$ induces an elliptic fibration on
$S_{5,5}$ with two reducible fibers: one of type $I_6^*$ whose
components are $Q_0$, $\Omega_{0,4}^{(1)}$, $\Omega_0^{(1)}$,
$\Omega_{1,0}^{(1)}$, $\Omega_1^{(1)}$, $\Omega_{2,1}^{(1)}$,
$\Omega_2^{(1)}$, $\Omega_{3,2}^{(1)}$, $\Omega_3^{(1)}$, $Q_3$,
and $\Omega_{4,3}^{(1)}$, one of type $III^*$ whose components are
$Q_4$, $\Omega_{3}^{(2)}$, $\Omega_{3,2}^{(2)}$,
$\Omega_{4,3}^{(2)}$, $\Omega_2^{(2)}$, $\Omega_{4}^{(2)}$,
$\Omega_{2,1}^{(2)}$, and $\Omega_{0,4}^{(2)}$.

The Weierstrass equation is \begin{equation}\label{eq: w B3 on S}y^2=x^3-\frac{\tau^3(\tau^3+6\tau^2+9\tau+3)}{3}x-\frac{\tau^5(\tau+3)(2\tau^3+12\tau^2+18\tau+9)}{27}.\end{equation}  So the discriminant $\Delta(\tau)$ is $-\tau^9(\tau+4)(\tau+1)^2$. Hence, by Tate's algorithm, this fibration has one fiber of type $I_6^*$ over $\tau=\infty$, one fiber of type $III^*$ over $\tau=0$, one fiber of type $I_2$ over $\tau=-1$ and one fiber of type $I_1$ over $\tau=-4$.


\subsubsection{The K3 surface $X_{5,5}$}\label{subsec: ef induced by cb on X}
Here we consider the elliptic fibrations induced by $B_i$ on $X_{5,5}$. We recall that in this case one has to apply the algorithm to (\ref{eq: X55 2:1 cover of P2}).
We summarize the results in the following table, where 
$r$ denotes the rank of the Mordell-Weil group of the fibration.

\begin{equation}
\label{eq: table fibrations from conic bundle generic}
\begin{array}{|c|c|c|c|}
\hline
&\Delta&singular fibers&r\\
\hline
B_1&\small{\begin{array}{c}-\tau^6(\tau-1)^6(\mu_1-\mu_2)^4\\
(\tau^3-2\tau^2-2\tau^2\mu_2+6\tau\mu_2+\tau+\tau\mu_2^2-4\mu_2)\\
(\tau^3-2\tau^2-2\tau^2\mu_1+6\tau\mu_1+\tau+\tau\mu_1^2-4\mu_1)\end{array}}&
\begin{array}{rr}1I_0^*&\tau=0\\ 2I_6&\tau=\infty,1\\6I_1\end{array}&2\\
\hline
B_2&\small{\begin{array}{c}
-\tau^8(-\tau^3+2\tau^2+\tau^2\mu_2+2\tau\mu_2-\tau+\mu_2)(-\tau+\mu_2)\\
(-\tau^3+2\tau^2+\tau^2\mu_1+2\tau\mu_1-\tau+\mu_1)(-\tau+\mu_1)(\mu_1-\mu_2)^4\end{array}}&
\begin{array}{rr}2I_8&\tau=0,\infty\\ 8I_1\end{array}&2\\
\hline
B_3&\small{\begin{array}{c}
-\tau^8(\tau+2\tau^2+\tau^3-6\mu_2\tau-2\mu_2\tau^2+4\mu_2^2+\tau\mu_2^2)\\
(\tau^3-6\mu_1\tau+2\tau^2+\tau+\mu_1^2\tau+4\mu_1^2-2\mu_1\tau^2)(\mu_1-\mu_2)^4
\end{array}}&
\begin{array}{rr}I_{10}&\tau=\infty\\ I_2^*&\tau=0\\6I_1\end{array}&1\\
\hline
\end{array}
\end{equation}

\subsubsection{Other K3 surfaces}
As we saw above, all the equations for K3 surfaces that are
double covers of $R_{5,5}$ branched over some special fibers can be
obtained from the
general equation for $X_{5,5}$ by substituting particular values of $\mu_1$ and $\mu_2$.
In particular in order to find the Weierstrass equations of the elliptic fibrations induced by the conic bundles $B_i$ on a specific K3 surface it suffices to substitute in \eqref{eq: table fibrations from conic bundle  generic} the appropriate values of $\mu_1$ and $\mu_2$.
As an example here we construct a table analogous to \eqref{eq: table fibrations from conic bundle  generic} if the K3 surface is obtained by $R_{5,5}$ branching along one fiber of type $I_5$ and one smooth fiber. We already noticed that one has two different choices for the $I_5$ branch fiber. Once one chooses the branch fiber $I_5$, the construction is not symmetric in $I_5$. 
For $i=1,2$ the reducible fibers of the conic bundle $|B_i|$ are symmetric up to switching the $I_5$-fibers, so the conic bundle $|B_i|$ is associated to elliptic fibrations with the same property choosing differently the $I_5$ branch fiber. For the conic bundles $|B_1|$ and $|B_2|$ we will choose the $I_5$ branch fiber to be the one over $\mu_1=0$.
The conic bundle $|B_3|$ has a unique reducible fiber, supported over one specific $I_5$, so the elliptic fibrations induced by this conic bundle have not necessarily the same properties if one change the $I_5$ branch fibers. So we give the equations of the elliptic fibrations if one chose both $\mu_1=0$ and $\mu_1=\infty$.

For all the conic bundles $\mu_2$ is the parameter of the 1-dimensional family of K3 surfaces we are considering. So we obtain the following (where $r=\rk(MW)$ and both the reducible fibers and $r$ are given for generic choice of $\mu_2$):

\begin{equation}\label{eq: table fibrations from conic bundle I_5I_0}\begin{array}{|c|c|c|c|}
\hline
&\Delta&singular fibers&r\\
\hline
B_1&\small{\begin{array}{c}-\tau^7\mu_2^4(-1+\tau)^8(\tau^3-2\tau^2-2\tau^2\mu_2+\tau+6\tau\mu_2+\tau\mu_2^2-4\mu_2)\end{array}}&
\begin{array}{rrrr}I_1^*&\tau=0,& I_2^*&\tau=1\\I_6&\tau=\infty,&3I_1\end{array}&1\\
\hline
B_2&\small{\begin{array}{c}
-\tau^{10}\mu_2^4(-1+\tau)^2(\mu_2+\tau^2\mu_2+2\tau\mu_2-\tau^3-\tau+2\tau^2)(\mu_2-\tau)\end{array}}&
\begin{array}{rrrr}I_4^*&\tau=0,&I_2& \tau=1\\I_8,&\tau=\infty,& 4I_1\end{array}&1\\
\hline
\begin{array}{c}
B_3,\\ \mu_1=0\end{array}&\small{\begin{array}{c}
-\mu_2^4\tau^9(\tau+1)^2(\tau+2\tau^2+\tau^3-6\mu_2\tau-2\mu_2\tau^2+4\mu_2^2+\tau\mu_2^2)
\end{array}}&
\begin{array}{rrrr}III^*&\tau=\infty,& I_2&\tau=-1\\I_{10}&\tau=0,&3I_1\end{array}&1\\
\hline
\begin{array}{c}
B_3,\\ \mu_1=\infty\end{array}&
-\tau^8(4+\tau)(\tau+2\tau^2+\tau^3-6\mu_2\tau-2\mu_2\tau^2+4\mu_2^2+\tau\mu_2^2)
&
\begin{array}{rrrr}I_{6}^*&\tau=0\\ I_2^*&\tau=\infty,&4I_1\end{array}&1\\
\hline
\end{array}
\end{equation}

\section{The K3 surface $S_{5,5}$ and its elliptic fibrations}\label{sec: the K3 surface S55}
The aim of this section is to prove the following results,
computing the equations of all the elliptic fibrations on
$S_{5,5}$.

\begin{proposition}\label{prop: classification elliptic fibrations S55}
The K3 surface $S_{5,5}$ admits 13 different elliptic fibrations.
One of them is induced by $\E_R$, 3 are induced by conic bundles,
3 by splitting genus one pencils and 6 by generalized conic
bundles. The equations of these elliptic fibrations are given in
\ref{eq: Weierstrass S, induced by R} (the one induced by $\E_R$),
in \ref{eq: w B1 on S}, \ref{eq: w B2 on S}, \ref{eq: w B3 on S}
(the ones induced by the conic bunldes), in \ref{eq: splitting
genus 1 pencils} (the ones induced by splitting genus 1 pencil)
and in \ref{eq: Weierstrass generalized conic bunde S}, \ref{eq:
N7} (the ones induced by generalized conic
bundles).\end{proposition}

To prove this, we deeply analyze the elliptic fibrations on $S_{5,5}$ induced by splitting genus $1$ pencils and generalized conic bundles, in particular giving an algorithm to find the Weierstrass equation of any elliptic fibration induced by a splitting genus 1 pencil.


The K3 surface $S_{5,5}$ can be also described as the (unique!) K3
surface which admits a non-symplectic involution fixing 10
rational curves. Indeed, by our construction it is clear that
$\iota$ fixes 10 rational curves (the inverse image of the
components of the two $I_5$ fibers). The fact that this K3 surface
is unique is classically known, and due to Nikulin, see
\cite{Nik}. The transcendental lattice of this K3 surface is
$T_S\simeq \langle 2\rangle\oplus \langle 2\rangle$. The elliptic
fibrations on this K3 surface are classified by Nishiyama, see
\cite[Table 1.2]{Nish}, who used a lattice theoretic method that
we will apply later to a different K3 surface, in Section \ref{sec: the K3 surface X}. The complete list of the elliptic fibrations is the
following:
\begin{equation}\label{table ell fib on S}
\begin{array}{lr}
\begin{array}{|c|c|c|}
\hline
\mbox{n}^o&\mbox{singular fibers}&MW\\
\hline
1&2II^*+2I_2&\{1\}\\
\hline
2&II^*+I_6^*+2I_1&\{1\}\\
\hline
3&I_{12}^*+2I_2+2I_1&\Z/2\Z\\
\hline
4&2III^*+I_0^*&\Z/2\Z\\
\hline
5&III^*+I_6^*+I_2+I_1&\Z/2\Z\\
\hline
6&I_{18}+I_2+4I_1&\Z/3\Z\\
\hline
\end{array}&
\begin{array}{|c|c|c|}
\hline
\mbox{n}^o&\mbox{singular fibers}&MW\\
\hline
7&I_{14}^*+4I_1&\{1\}\\
\hline
8&I_8^*+I_2^*+2I_1&\Z/2\Z\\
\hline
9&2I_4^*+2I_2&\left(\Z/2\Z\right)^2\\
\hline
10&I_{16}+I_4+4I_1&\Z/4\Z\\
\hline
11&IV^*+I_{12}+4I_1&\Z\times\Z/3\Z\\
\hline
12&3I_2^*&\left(\Z/2\Z\right)^2\\
\hline
13&2I_{10}+4I_1&\Z/5\Z\\
\hline
\end{array}\end{array}
\end{equation}

Since the involution $\iota$ acts as the identity on the
N\'eron--Severi group of $S_{5,5}$, there are no fibrations of type 3 on
this K3 surface.

The fibration 13 in Table \eqref{table ell fib on S} is the one
induced by the fibration $\E_{R_{5,5}}$ and it has equation
\eqref{eq: Weierstrass S, induced by R}. By Section \ref{sec:
elliptic fibrations induced by conic bundles}, the fibrations 5, 9,
and 12 are induced by conic bundles and their equations are
\eqref{eq: w B3  on S}, \eqref{eq: w B2  on S}, and \eqref{eq: w
B1 on  S}, respectively.

The other fibrations are induced either by generalized conic
bundles or by splitting genus 1 pencils.

An elliptic
fibration induced by a splitting genus 1 pencil corresponds to a fibration of genus 1 curves on a non-relatively minimal rational elliptic surface (that is, this fibration admits $(-1)$-curves as components of some fibers). The original relatively minimal rational elliptic surface $R_{5,5}$ can be recovered from this non-minimal surface by blowing down some divisors. A different choice of divisors to blow-down, namely the $(-1)-$curves which are components of the fibers of the splitting genus 1 pencil, gives us another rational elliptic surface on which the splitting genus 1 pencil above corresponds to a relatively minimal elliptic fibration. Hence each fibration given by a splitting genus 1 pencil is indeed induced by a rational elliptic surface (different from
$R_{5,5}$) by a base change of order 2 whose branch locus
consists of 10 curves. Considering the list of elliptic
fibrations on $S_{5,5}$ given in Table \ref{table ell fib on S} one observes immediately that the fibrations 6,
10 and 11 could be of this type, i.e., they could be induced by
splitting genus 1 pencils, (this is in fact proved in \cite{GS}). Indeed they present some fibers which appear in pairs (each pair is good candidate to be the inverse image of a unique fiber on the rational elliptic surface) and at most two other fibers, which do not appear an even number of time, but which are either of type  $I_{2n}$ or of type $IV^*$. These fibers are the one obtained by base change of order 2 branched over $I_n$ or $IV$-fibers, so they could be the ramification fibers of the base change. The others fibration in Table \ref{table ell fib on S} have not the same properties, thus can not be induced by rational elliptic surface by a base change of order 2.
We already observed that the fibration
13 is induced by the elliptic fibration on $R_{5,5}$ and that the
fibrations 5, 9 and 12 are induced by conic bundles, so the fibrations 1, 2, 3, 4, 7, 8 are induced by generalized conic bundles.

\subsection{Splitting genus 1 fibrations}


\subsubsection{An example, the fibration 6}
We give an example of splitting genus 1 pencil of curves: we are
looking for a fiber of type $I_{18}$ and it is given by
$$\begin{array}{lll} W_X&:=&Q_0+\Omega_0^{(2)}+\Omega_{4,0}^{(2)}+\Omega_{4}^{(2)}+\Omega_{4,3}^{(2)}+\Omega_3^{(2)}+\Omega_{3,2}^{(2)}+\Omega_{2}^{(2)}+\Omega_{2,1}^{(2)}+\Omega_1^{(2)}+\\
&&+Q_3+\Omega_3^{(1)}+\Omega_{3,2}^{(1)}+\Omega_2^{(1)}+\Omega_{2,1}^{(1)}+\Omega_1^{(1)}+\Omega_{1,0}^{(1)}+\Omega_0^{(1)}.\end{array}$$

Using \eqref{tildeR curve ids}, the class $W_R \colonequals \pi_*W_X$ is
\begin{multline}2F_1+E_1+2G_1+\ell_1+2H_2+E_2+2G_2+\ell_3+2E_{T_1}+\ell_2+2F_4+E_4+2G_4+m_3+2H_3+E_3+2G_3+m_1=\\
5h-2E_1-2F_1-2G_1-4H_1-E_2-2F_2-G_2-H_2-2E_3-4F_3-2G_3-2H_3-E_4-F_4-G_4-2H_4-2E_{Q_5}-E_{T_2}.\end{multline}
Since we know that the curves in the linear system described
split in the double cover, we can assume that the class $W_R$ is
both the push down and the geometric image of a fiber of our
fibration (i.e. $\pi_*(W_X)=\pi(W_X)=W_R$).

This class is the strict transform on $\widetilde{R}$ of quintics in $\mathbb{P}^2$ with the following properties:
they have a tacnode in $Q_1$ with principal tangent $\ell_2$;
they have a tacnode in $Q_3$ with principal tangent $\ell_3$;
they have a node in $Q_5$;
the tangent in $Q_2$ is $m_2$;
the tangent in $Q_4$ is $m_2$;
they pass through $T_2$.

This gives the following families of
quintics
\begin{multline}\label{eq: type 2, ok}-bx_0^5+bx_0^4x_1+bx_0^4x_2-ex_0^3x_1x_2+ex_0^2x_1^2x_3+ex_0^2x_1x_2^2+\\(-3b-e)x_0x_1^2x_2^2+bx_1^3x_2^2+bx_1^2x_2^3 =0.\end{multline}

This is the equation of the splitting genus 1 pencil that we are
looking for. It indeed corresponds to a pencil of curves of genus 1 parametrized by $(b:e)$.

We now have to intersect the branch sextic given in equation
\eqref{eq: S 2:1 cover of  P2} with \eqref{eq: type 2,
ok}. The resultant of the polynomials \eqref{eq: S 2:1 cover of  P2} and \eqref{eq: type 2, ok} with respect to the variable
$x_0$ is
$$-x_1^{12}x_2^{12}b^3(x_1-x_2)^4(x_2+x_1)^2(e+2b).$$
We observe that all the solutions in $(x_1,x_2)$ have even
multiplicities, as is necessarily if the double cover splits into two
curves, isomorphic to the base curve, see \cite[Section 3]{GS}.
More precisely, the curve splits after the base change of order two branched in $(b:e)=(0:1)$ and $(b:e)=(1:-2)$,  cf. Lemma \ref{lem: branch} below.
In order to write down explicitly the Weierstrass form of the
elliptic fibration on $X_{5,5}$ one first finds the rational
elliptic fibration given by the splitting pencil of genus 1 curves
\eqref{eq: type 2, ok}, and then one performs the base change of
order two.
This can be computed by any computer algebra system and it is
$$y^2=x^3+A(b:e)x+B(b:e),$$ where
$$A(b:e):=\frac{23}{48}b^4-\frac{5}{12}b^3e-\frac{1}{8}b^2e^2+\frac{1}{12}be^3-\frac{1}{48}e^4$$
and
$$B(b:e):=-\frac{181}{864}b^6-\frac{17}{144}b^5e+\frac{31}{288}b^4e^2-\frac{1}{54}b^3e^3-\frac{5}{288}b^2e^4+\frac{1}{144}be^5-\frac{1}{864}e^6.$$

The discriminant of this rational elliptic fibration is
$$\frac{1}{16}b^9(e+2b)(e^2-5be+13b^2)$$
and the fibration has one fiber of type $I_9$ and three other
singular fibers, all of type $I_1$.

Now we consider the base change of order 2 branched in
$(b:e)=(0:1)$ and $(b:e)=(1:-2)$. It can be directly written as a map
$\mathbb{P}^1_{(\beta:\epsilon)}\ra\mathbb{P}^1_{(b:e)}$, where
$b=\beta^2$ and $e=-\beta^2+2\beta\epsilon+\epsilon^2$.

So we obtain a new elliptic fibration on $S_{5,5}$ whose equation is
\begin{equation}\label{eq: Weierstrass form type 2 on
S}y^2=x^3+A'(\beta:\epsilon)x+B'(\beta:\epsilon),\end{equation}
$$A'(\beta:\epsilon):=\frac{23}{48}\beta^8-\frac{5}{12}\beta^6(-\beta^2+2\beta\epsilon+\epsilon^2)-\frac{1}{8}\beta^4(-\beta^2+2\beta\epsilon+\epsilon^2)^2+$$
$$+\frac{1}{12}\beta^2(-\beta^2+2\beta\epsilon+\epsilon^2)^3-(1/48)(-\beta^2+2\beta\epsilon+\epsilon^2)^4$$
and
$$B'(\beta:\epsilon):=-(1/108)\beta^{12}-(5/9)\beta^{11}\epsilon-(7/18)\beta^{10}\epsilon^2+(28/27)\beta^9\epsilon^3+(7/12)\beta^8\epsilon^4
-(11/12)\beta^7\epsilon^5 +$$
$$-(35/72)\beta^6\epsilon^6+(1/3)\beta^5\epsilon^7+(5/24)\beta^4\epsilon^8-(5/108)\beta^3\epsilon^9-(1/18)\beta^2\epsilon^{10}-(1/72)\epsilon^{11}\beta-(1/864)\epsilon^{12}.$$
The discriminant is
$$(1/16)\beta^{18}(19\beta^4-14\beta^3\epsilon-3\beta^2\epsilon^2+4\beta\epsilon^3+\epsilon^4)(\epsilon+\beta)^2.$$
This fibration has a fiber of type $I_{18}$, as expected, one
fiber of type $I_2$ (in $(\beta:\epsilon)=(1:-1)$), and four fibers of
type $I_1$. By construction this elliptic fibration exhibits $S_{5,5}$
as the double cover of a rational elliptic surface with
one fiber of type $I_9$ and three fibers of type $I_1$.

\subsubsection{Splitting genus 1 fibration: an algorithm}\label{subsec: algorithm splitting genus 1}
The aim of this section is to generalize the previous construction
to other splitting genus 1 pencils.

\noindent \textbf{Setup.} Let $\pi \colon V \to \PP^1$ be a K3 surface which is a double cover of a (not necessarily minimal) rational elliptic surface $\tilde{R}$ branched over two fibers.  If $\mathcal{H} \colon \tilde{R} \to \PP_{(b:e)}^1$ is a splitting genus $1$ pencil, then the induced elliptic fibration on $V$ comes via pullback from a double cover $\PP^1_{(\beta: \epsilon)} \to \PP_{(b:e)}^1$.  Hence given an equation for the fibration $\mathcal{H}$, it suffices to find the branch points of the map $\PP^1_{(\beta: \epsilon)} \to \PP_{(b:e)}^1$ in order to find the Weierstrass equation for the elliptic fibration on $V$.  We now explain how to do this in general, when the branch curves and equations for $\mathcal{H}$ are given in $\PP^2$.

Assume that the equation of $V$ as double
cover of $\mathbb{P}^2$ is given by
$w^2=f_3(x_0:x_1:x_2)g_3(x_0:x_1:x_2)$. Let
$h((x_0:x_1:x_2),(b:e))$ be the equation of the pushforward of a splitting genus 1 pencil from $\tilde{R}$ to $\PP^2$.
The equation $h$ is homogeneous of some degree in the coordinate $(x_0:x_1:x_2)$ on $\PP^2$ and linear in the coordinate $(b:e)$ of the base $\PP^1$ of the pencil.  For every $(b:e)$, the curve with equation $h((x_0:x_1:x_2),(b:e))$ is of arithmetic genus $1$.

Write $\cup_{k\in K} C_k$ for the irreducible components of the branch curves $f_3g_3 = 0$.  For $D \subset \pp^2$, write $\text{mult}_{C_k}(D)$ for the multiplicity of the component $C_k$ in $D$. Then we have the following.
\begin{lemma}\label{lem: branch}
A plane curve $D_{(b:e)} \subset \PP^2$ that is member of the pencil $\mathcal{H}$ is a branch curve for the double cover
\begin{center}
\begin{tikzcd}
V \arrow{d} \arrow{r} & \tilde{R} \arrow{d}{\mathcal{H}} \\
\PP^1_{(\beta:\epsilon)} \arrow{r}{2:1} & \PP^1_{(b:e)}
\end{tikzcd}
\end{center}
induced by the splitting genus 1 pencil
if and only if there exists $k \in K$ such that $\text{mult}_{C_k}(D_{(b:e)}) \neq 0$.
\end{lemma}
\begin{proof}
If $D$ meets the branch curves transversely, it does so only in points of even multiplicity, and so splits in the double cover as two disjoint elliptic curves.  It suffices, therefore, to show that if $D$ contains at least one component $C_k$, the support of the preimage of $D$ under the $2:1$ map $\pi \colon V \to \PP^2$ is connected.  This follows from the fact that the preimage of $C_k$ is a double curve, the support of which maps isomorphically onto $C_k$, and so every component of $\pi^*D$ must meet this curve.
\end{proof}

We will also make use of the following elementary fact.  Let $D_{(b_0:e_0)} \subset \PP^2$ be a plane curve in the pencil $\mathcal{H}$ with equation $h((x_0:x_1:x_2),(b_0:e_0))$.  Denote by $r((x_1:x_2)(b:e))$ the
resultant of $f_3(x_0:x_1:x_2)g_3(x_0:x_1:x_2)$ and $h((x_0:x_1:x_2),(b:e))$ with respect to $x_0$.
\begin{lemma}
The resultant $r((x_1:x_2)(b:e))$ vanishes to order
\[=\sum_k \operatorname{mult}_{C_k}(D_{(b_0:e_0)}) \cdot \begin{cases} \deg C_k &: \  (1:0:0) \not\in C_k \\ \left( \deg C_k - 1\right) & : \  (1:0:0) \in C_k\end{cases} \]
at $(b:e) = (b_0:e_0)$.
\end{lemma}
\begin{proof}
This follows from the geometric description of the zeros of the resultant in terms of projecting the scheme-theoretic intersection of $h=0$ and $f_3g_3=0$ from the point $(1:0:0)$.
\end{proof}
Writing $\operatorname{ord}_{(b_0:e_0)}$ for the order of vanishing at $(b_0:e_0)$, we may combine these two Lemmas to show the following:

\begin{corollary}
If $\operatorname{ord}_{(b_0:e_0)}\big(r((x_1:x_2)(b:e))\big)>0,$ then the curve $D_{(b_0:e_0)}$ is a branch curve for the double cover induced by the splitting genus 1 pencil.

If $(1:0:0) \not\in \cup_k C_k$ and $D_{(b_0:e_0)}$ is a branch curve for the double cover induced by the splitting genus 1 pencil, then  $\operatorname{ord}_{(b_0:e_0)}\big(r((x_1:x_2)(b:e))\big)>0.$
\end{corollary}

We can therefore determine the relevant branch points from the resultant.  This leads to the following algorithm.


\noindent\textbf{Algorithm.}\begin{enumerate}
\item Compute the resultant $r((x_1:x_2)(b:e))$ of the two polynomials
$f_3(x_0:x_1:x_2)g_3(x_0:x_1:x_2)$ and $h((x_0:x_1:x_2),(b:e))$ in
one variable, say $x_0$.
\item Observe that
$r((x_1:x_2)(b:e))=c_1(b:e)c_2(b:e)s((x_1:x_2),(b:e))^2$, where
$c_i(b:e)$ are homogeneous polynomials each with a unique root,
denoted by $(b_i:e_i)$. (If $(1:0:0)$ is in the branch curves, it may be necessary to change coordinates first.)
\item Write the Weierstrass form of
$h((x_0:x_1:x_2),(b:e))$, by applying the standard
transformations. This is the equation of a rational elliptic
surface, and the base of the fibration is $\mathbb{P}^1_{(b:e)}$.
\item Consider the base change
$\mathbb{P}^1_{(\beta:\epsilon)}\ra\mathbb{P}^1_{(b:e)}$ given by
$(b=\beta^2(b_1/e_1)+\epsilon^2,e=\beta^2+(e_2/b_2)\epsilon^2)$ (cf. \eqref{eq: base change}). Substituting this base change in the
previous Weierstrass equation, one finds the Weierstrass equation
of the elliptic fibration on the K3 surface $V$ whose base is
$\mathbb{P}^1_{(\beta:\epsilon)}$.
\end{enumerate}
\subsection{The elliptic fibrations on $S_{5,5}$}
In this section we want to describe all the elliptic fibrations on
$S_{5,5}$ giving  equations for each of them.

In \cite{CG} a model of $S_{5,5}$ as a double
cover of $\mathbb{P}^2$ was given and the elliptic fibrations
induced by (generalized) conic bundles are already studied geometrically in that
context. Here we explicitly describe in our context both the fibrations induced by generalized conic bundles and the ones induced by splitting genus 1 curves, giving also a Weierstrass equation for each of them.

\subsubsection{Elliptic fibrations induced by splitting genus 1 pencils}
The fibration 11 of Table \eqref{table ell fib on S} is induced by
the class of the fiber
$$M_1:=Q_0+\Omega_0^{(1)}+\Omega_{1,0}^{(1)}+\Omega_1^{(1)}+Q_1+\Omega_2^{(2)}+\Omega_{3,2}^{(2)}+\Omega_3^{(2)}+\Omega_{4,3}^{(2)}+\Omega_4^{(2)}+\Omega_{4,0}^{(2)}+\Omega_0^{(2)}$$
which is the class of a fiber of type $I_{12}$. The curves $\Omega_3^{(1)}$, $\Omega_{3,2}^{(1)}$, $\Omega_2^{(1)}$, $\Omega_{3,4}^{(1)}$, $\Omega_4^{(1)}$, $Q_3$, and $\Omega_1^{(2)}$ are orthogonal to the components of the $I_{12}$-fiber and form a fiber of type $IV^*$. The classes $Q_2$, $Q_4$, $\Omega_{1,2}^{(1)}$, $\Omega_{4,0}^{(1)}$, $\Omega_{1,0}^{(2)}$, $\Omega_{2,1}^{(1)}$ are sections of the elliptic fibration induced by $|M_1|$. We observe that there are 6 curves fixed by $\iota$ among the components of the fiber of type $I_{12}$ (the curves $\Omega_i^{(j)}$ for $(i,j)=(0,1)$, $(1,1)$, $(2,2)$, $(3,2)$, $(4,2)$, $(0,2)$) and 4 among the components of the $IV^*$-fiber (the curves $\Omega_j^{(j)}$ for $(i,j)=(3,1)$, $(2,1)$, $(4,1)$, $(1,2)$).

The push-down of the class $M_1$ is $$3h-E_1-F_1-G_1-2H_1-E_2-2F_2-G_2-H_2-E_3-F_3-G_3-H_3-E_{T_1}-E_{T_2}-E_5$$
which corresponds to a pencil of cubics passing through $Q_1$ with tangent line $\ell_2$, through $Q_2$ with tangent line $m_2$, through $Q_3$, $E_{T_1}$, $E_{T_2}$, and $Q_5$.
The equation of this pencil is
\begin{equation}\label{eq: splitting genus 1 fib 11} b(x_0^2x_1-x_0x_1^2+x_1^2x_2+x_1x_2^2-2x_0^2x_2)+e(x_0x_1x_2-x_0^2x_2)=0.\end{equation}

The Weierstrass form of the (rational)
elliptic fibration associated to the pencil \eqref{eq: splitting
genus 1 fib 11} is
$$y^2=x^3-\frac{1(28b^2+4eb+e^2)(2b+e)^2}{48}x-\frac{(376b^4+176eb^3+60e^2b^2+8e^3b+e^4)(2b+e)^2}{864},$$
whose discriminant is $b^6(31b^2+4eb+e^2)(2b+e)^4/16$. This
elliptic fibration has a fiber of type $I_6$ on $b=0$ and one of
type $IV$ on $(b:e)=(1:-2)$.

%

The fibration 10 of Table \eqref{table ell fib on S} is induced by the class of the fiber
$$M_2:=Q_0+\Omega_0^{(1)}+\Omega_{1,0}^{(1)}+\Omega_1^{(1)}+\Omega_{2,1}^{(1)}+\Omega_{2}^{(1)}+\Omega_{3,2}^{(1)}+\Omega_3^{(1)}+\Omega_{4,3}^{(1)}+\Omega_4^{(1)}+Q_4+\Omega_3^{(2)}+\Omega_{4,3}^{(2)}+\Omega_4^{(2)}+\Omega_{4,0}^{(2)}+\Omega_0^{(2)},$$
which is the class of a fiber of type $I_{16}$.

The push-down of the class $M_2$ is $$4h-E_1-F_1-G_1-2H_1-E_2-F_2-2G_2-H_2-E_3-2F_3-G_3-H_3-E_4-2F_4-G_4-H_4-2E_{T_2}-2E_5$$
which corresponds to a pencil of quartics with bitangent lines $l_2$ (in $Q_2$ and $Q_4$) and $l_3$ (in $Q_1$ and $Q_3)$ and having two nodes in $E_{T_2}$ and $E_{5}$.
The equation of this pencil is
\begin{multline}\label{eq: splitting genus 1 fib 10}
bx_0^4-2bx_0^3x_1+bx_0^2x_1^2-2bx_0^3x_2+(3b+4e)x_0^2x_1x_2+(-b-4e)x_0x_1^2x_2+ \\ ex_1^3x_2+bx_0^2x_2^2+(-b-4e)x_0x_1x_2^2+2ex_1^2x_2^2+ex_1x_2^3 =0.\end{multline}

Using the algorithm, we find the following Weierstrass equations for elliptic fibrations on $S_{5,5}$ induced by splitting genus $1$ pencils.

\begin{equation}\label{eq: splitting genus 1 pencils}
\begin{array}{|c|c|} \hline
i&\mbox{elliptic fibrations }\\ \hline
6 & \begin{array}{l} A = \frac{23}{48}\beta^8-\frac{5}{12}\beta^6(-\beta^2+2\beta\epsilon+\epsilon^2)-\frac{1}{8}\beta^4(-\beta^2+2\beta\epsilon+\epsilon^2)^2+\\
\qquad +\frac{1}{12}\beta^2(-\beta^2+2\beta\epsilon+\epsilon^2)^3-(1/48)(-\beta^2+2\beta\epsilon+\epsilon^2)^4 \\ B=-(1/108)\beta^{12}-(5/9)\beta^{11}\epsilon-(7/18)\beta^{10}\epsilon^2+(28/27)\beta^9\epsilon^3+(7/12)\beta^8\epsilon^4
-(11/12)\beta^7\epsilon^5 + \\
\qquad -(35/72)\beta^6\epsilon^6+(1/3)\beta^5\epsilon^7+(5/24)\beta^4\epsilon^8-(5/108)\beta^3\epsilon^9-(1/18)\beta^2\epsilon^{10}-(1/72)\epsilon^{11}\beta-(1/864)\epsilon^{12} \\
\Delta =  (1/16)\beta^{18}(19\beta^4-14\beta^3\epsilon-3\beta^2\epsilon^2+4\beta\epsilon^3+\epsilon^4)(\epsilon+\beta)^2
\end{array} \\ \hline
11 &  \begin{array}{l}A = -(24\epsilon^4+\beta^4)\beta^4/48, \
B = -(216\epsilon^8+36\beta^4\epsilon^4+\beta^8)\beta^4/864  \\
\Delta = \epsilon^{12}(27\epsilon^4+\beta^4)\beta^8/16
\end{array}\\ \hline
10 & \begin{array}{l} A = (-\beta^8+16\epsilon^4\beta^4-16\epsilon^8)/48, \
B = -(\beta^4-8\epsilon^4)(-8\epsilon^8-16\epsilon^4\beta^4+\beta^8)/864 \\
\Delta =\epsilon^{16}\beta^4(2\epsilon-\beta)(2\epsilon+\beta)(4\epsilon^2+\beta^2)/16 \end{array} \\ \hline
\end{array}
\end{equation}

\subsubsection{Elliptic fibrations induced by generalized conic
bundles}
Let us consider the divisors:
$$\begin{array}{lll}N_1:&=&2\Omega_0^{(1)}+4Q_0+6\Omega_0^{(2)}+5\Omega_{1,0}^{(2)}+4\Omega_{1}^{(2)}+3\Omega_{2,1}^{(2)}+2\Omega_{2}^{(2)}+\Omega_{3,2}^{(2)}+3\Omega_{4,0}^{(2)},\\N_2:&=&2\Omega_4^{(2)}+4\Omega_{4,0}^{(2)}+6\Omega_0^{(2)}+5\Omega_{1,0}^{(2)}+4\Omega_{1}^{(2)}+3\Omega_{2,1}^{(2)}+2\Omega_{2}^{(2)}+\Omega_{3,2}^{(2)}+3Q_0,\\
N_3:&=&Q_0+\Omega_{4,0}^{(1)}+Q_4+\Omega_{4,3}^2+2\left(\Omega_0^{(1)}+\Omega_{1,0}^{(1)}+\Omega_1^{(1)}+\Omega_{2,1}^{(1)}+\right.\\&+&\left.\Omega_2^{(1)}+\Omega_{3,2}^{(1)}+\Omega_{3}^{(1)}+Q_3+\Omega_1^{(2)}+\Omega_{2,1}^{(2)}+\Omega_2^{(2)}+\Omega_{3,2}^{(2)}+\Omega_3^{(2)}\right),\\
N_4:&=&\Omega_{1,0}^{(2)}+2\Omega_0^{(2)}+3Q_0+4\Omega_0^{(1)}+3\Omega_{1,0}^{(1)}+2\Omega_1^{(1)}+\Omega_{2,1}^{(1)}+2\Omega_{4,0}^{(1)},\\N_8:&=&\Omega_{1,0}^{(1)}+Q_0+\Omega_{4,0}^{(2)}+\Omega_{4,3}^{(2)}+2\left(\Omega_0^{(1)}+\Omega_{4,0}^{(1)}+\Omega_{4}^{(1)}+\Omega_{4,3}^{(1)}+\Omega_{3}^{(1)}+\Omega_{3,2}^{(1)}+\Omega_{2}^{(1)}+Q_2+\Omega_{4}^{(2)}\right).\end{array}$$

The linear system $|N_i|$ induces the elliptic fibration number
$i$ of the Table \eqref{table ell fib on S}, indeed the divisor
$N_i$ corresponds to an elliptic fibration with a fibre of type
$II^*$, $II^*$, $I_{12}^*$, $III^*$, $I_8^*$ if $i=1,2,3,4,8$
respectively, so $N_3$ (resp. $N_8$) corresponds to the unique fibration in Table \eqref{table ell fib on S} with a fiber of type $I_{12}^*$ (resp. $I_8^*$), i.e. the fibration 3 (resp. 8). A priori, $N_1$ could correspond either to the fibration 1 or to the fibration 2 (which are the fibrations which admit at least a fiber of type $II^*$). To distinguish among these cases we consider the other reducible fibers: the curves orthogonal to $N_1$ are
$\Omega_{h,j}^{(1)}$, for $h,j\in\{1,2,3,4\}$, $h<j$,
$\Omega_{k}^{(1)}$ for $k=1,2,3,4$, $Q_2$ and $Q_4$ and they span
the lattice $\widetilde{E_8}$, so $N_1$ corresponds to the
fibration 1 in Table \eqref{table ell fib on S}; the curves
orthogonal to $N_2$ are $\Omega_{h,j}^{(1)}$, for
$h,j\in\{0,1,2,3,4\}$, $h<j$, $\Omega_{k}^{(1)}$ for $k=1,2,3,4$
and $Q_2$ and they span the lattice $D_{10}$, so $N_2$ corresponds
to the fibration 2 in Table \eqref{table ell fib on S}.

Let us now consider the fibration $N_4$. The fiber
associated to $N_4$ is a fiber of type $III^*$. The curves
$\Omega_2^{(1)}$ and $\Omega_{1}^{(1)}$ are sections of the
fibration. The curves $\Omega_{2,1}^{(2)}$, $\Omega_{2}^2$,
$\Omega_{3,2}^{(2)}$, $\Omega_3^{(2)}$, $\Omega_{4,3}^{(2)}$,
$\Omega_4^{(2)}$, $Q_2$, $Q_4$ are orthogonal to $N_4$ and are
the components of another fiber of type $III^*$. This implies that $|N_4|$ is the fibration 4 in Table \eqref{table ell fib on S}, and that the unique other reducible fiber is of type $I_0^*$. The curves $Q_3$,
$\Omega_3^{(1)}$, $\Omega_{4,3}^{(1)}$ and $\Omega_{3,2}^{(1)}$
are orthogonal to $N_4$ and span a lattice isometric to $D_4$. So
they are components of the $I_0^*$-fiber in the fibration $|N_4|$. A $I_0^*$-fiber has five components, four of them are  $Q_3$, $\Omega_3^{(1)}$,
$\Omega_{4,3}^{(1)}$ and $\Omega_{3,2}^{(1)}$, the fifth component
is another curve, say $V_1$. So that we have a special fiber of
the fibration $|N_4|$, which is
$2\Omega_3^{(1)}+\Omega_{4,3}^{(1)}+\Omega_{3,2}^{(1)}+Q_3+V_1$.
Hence we can express the class of the curve $V_1$ as
$N_4-(2\Omega_3^{(1)}+\Omega_{4,3}^{(1)}+\Omega_{3,2}^{(1)}+Q_3)$.
From this expression one can compute all the intersection
numbers of $V_1$ with all the curves $\Omega_{i,j}^k,
\Omega_m^n$ and $Q_k$. One can also observe that since $V_1$ is a
component of a fiber of the fibration $|N_4|$, it is orthogonal to
all the components of the two other reducible fibers of the same
fibration, i.e. to the components of the $III^*$-fibers.

Let us now consider the class:
$$\begin{array}{lll}N_7&:=&\Omega_{1,0}^{(1)}+\Omega_{4,0}^{(1)}+V_1+\Omega_{4,3}^{(1)}+2(\Omega_0^{(1)}+Q_0+\Omega_{0}^{(2)}+\Omega_{1,0}^{(2)}+\Omega_{1}^{(2)}+\Omega_{2,1}^{(2)}+\\
&+& \Omega_{2}^{(2)}+\Omega_{3,2}^{(2)}+\Omega_{3}^{(2)}+\Omega_{4,3}^{(2)}+\Omega_{4}^{(2)}+Q_2+\Omega_2^{(1)}+\Omega_{3,2}^{(1)}+\Omega_3^{(1)}).\end{array}$$
It is the class of the elliptic fibration 7 in the Table \eqref{table ell fib on S} since it is the class of a fiber of an elliptic  fibration on $S_{5,5}$ with a fiber of type $I_{14}^*$.

All the classes $N_i$, for $i=1,2,3,4, 7,8$ considered above correspond to generalized conic bundles which can be written in the following way in terms of the classes on $\widetilde{R}$:

$$
\begin{array}{|c|c|}
\hline
i&(\pi_*(N_i))/2\\
\hline
1&4h-(E_2+F_2+G_2+H_2)-(2E_3+3F_3+3G_3+2H_3)-2(E_4+2F_4+G_4+H_4)-E_{T_2}\\
\hline
2 &4h-(2E_2+2F_2+2G_2+3H_2)-2(E_4+2F_4+G_4+H_4)-(E_3+2F_3+G_3+H_3)-E_5\\
\hline
3&
4h-(2E_1+2F_1+2G_1+3H_1)-(E_4+F_4+G_4+2H_4)-2(E_3+2F_3+G_3+H_3)-E_5\\
\hline
4&2h-(E_1+F_1+2G_1+H_1)-(E_3+2F_3+G_3+H_3)\\
\hline
7& \begin{array}{l}7h-3(E_1+F_1+2G_1+H_1)-(E_2+2F_2+G_2+H_2)+\\-2(E_4+2F_4+G_4+H_4)-(4E_3+6F_3+4G_3+5H_3)\end{array}\\
\hline
8&\begin{array}{l}4h-2(E_+-F_1+G_1+H_1)-(2E_2+3F_2+2G_2+2H_2)+\\-(E_4+2F_4+G_4+H_4)-(2E_3+2F_3+2G_3+3H_3)\end{array}\\
\hline
\end{array}
$$


This allows to compute explicitly the equations of pencils of
singular rational curves in $\mathbb{P}^2$ corresponding to our
elliptic fibrations on $S_{5,5}$.

$$
\begin{array}{|c|c|}
\hline
i&\mbox{pencils in }\mathbb{P}^2\\
\hline
1&a(x_0^4-x_0^3x_1+3x_0^2x_1x_2-4x_0^3x_2+6x_0^2x_2^2-3x_0x_1x_2^2-4x_0x_2^3+x_1x_2^3+x_2^4)+g x_0x_1^2x_2\\
\hline
2&n(x_0^4-2x_0^3x_1+x_0^2x_1^2+7x_0^2x_1x_2-3x_1^2x_0x_2-4x_2x_0^3+6x_2^2x_0^2-8x_0x_1x_2^2)+\\&+n(2x_1^2x_2^2-4x_2^3x_0+3x_1x_2^3+x_2^4)+m(8x_1^2x_2x_0-8x_1^3x_2)\\
\hline
3&f(-x_0^4+x_0^2x_1x_2+x_0^3x_2-x_1^2x_2^2)+n(x_0^4+x_2^2x_1^2-x_0^3x_2-x_0x_1x_2^2)\\
\hline
4&\tau x_0^2+\sigma(x_0x_2-x_1x_2)\\
\hline 7& \begin{array}{l}\tau
x_0x_1x_2(x_2-x_0)(x_0-x_1)(x_0^2-x_0x_2+x_1x_2)
+\sigma(-x_1x_0^6+2x_0x_1^3x_2^3-12x_0^4x_1x_2^2+\\+7x_0^3x_1^2x_2^2-8x_0^2x_1^2x_2^3+6x_0^5x_1x_2-3x_0^4x_1^2x_2+
x_0^7-x_2^4x_1^3-x_0^2x_1^3x_2^2-3x_2^4x_1x_0^2+\\+10x_0^3x_1x_2^3+6x_2^2x_0^5-4x_2x_0^6+3x_2^4x_1^2x_0+x_2^4x_0^3-4x_2^3x_0^4)\end{array}\\
\hline
8&\begin{array}{l}s(x_0^4-2x_0^3x_1+x_0^2x_1^2+3x_0^2x_1x_2-x_0x_1^2x_2-2x_0^3x_2+x_0^2x_2^2-x_0x_1x_2^2)+\\+t(-x_0^2x_1x_2+x_0x_1^2x_2+2x_0x_1x_2^2-x_1^2x_2^2)\end{array}
\\
\hline
\end{array}$$

Now it remains to find the Weierstrass equations of the elliptic
fibrations on $S_{5,5}$ corresponding to the linear systems $\left(\pi_*(N_j)/2\right)$
on $\widetilde{R}$.  In case $4$, the curves which are fibers of
the conic bundle have degree 2, so we can directly apply the first
algorithm in Section \ref{subsec: an algorithm (generalized)
conic bundle}.  In cases $1,2,3,$ and $8$ the curves are quartics which satisfy condition
$(\dagger)$ and so we may apply the second algorithm in Section \ref{subsec: an algorithm (generalized)
conic bundle}.  The rational parameterizations and induced Weierstrass equations
are given by
\begin{equation}\label{eq: Weierstrass generalized conic bunde S}
{\scriptsize
\begin{array}{|c|c|c|}
\hline
i&\mbox{rational parameterization }&\mbox{elliptic fibration }\\ \hline
1 & \begin{array}{l} x_0 = -agp^3 - 2agp^2 - agp \\ x_1 = -g^2p^4 - 2agp^2 - a^2 \\ x_2 = g^2p^4 - agp^3 + g^2p^3 - agp^2 \end{array} & \begin{array}{c} A=- 3a^4g^4, B=a^7g^5 + a^5g^7 \\ \Delta=-432g^{10}  a^{10}(a - g)^2 (a + g)^2  \end{array} \\ \hline
2 & \begin{array}{l} x_0 = -64m^2p^4 + 8nmp^3 - 8nmp \\ x_1 = -64m^2p^4 + 16nmp^3 \\ \qquad - n^2p^2 - 16nmp^2 + 2n^2p - n^2 \\ x_2 = -64m^2p^4 + 8nmp^3 + 64m^2p^3 - 8nmp^2 \end{array} & \begin{array}{c} A=108m^2n^4(-n^2 + 384m^2) \\ B = -432m^3n^5 (n^4 - 576n^2m^2 + 55296m^4)\\ \Delta=570630428688384n^{10} m^{12} (n^2 - 432m^2)  \end{array} \\ \hline
3 & \begin{array}{l} x_0 = p (-np + fp + n)(-np^2 + fp^2 + fp - n + f) \\ x_1 = (p + 1) (-n + f) p^2 (-np + fp + n) \\ x_2 = (-np^2 + fp^2 + fp - n + f)^2 \end{array} & \begin{array}{c} A=(3f^6 - 18f^5n + 36f^4n^2 - 24f^3n^3 - 3f^2n^4 + 6fn^5 - n^6)\cdot \\ 27(f - n)^2  \\ B = (9f^6 - 54f^5n + 117f^4n^2 - 108f^3n^3 + 39f^2n^4 - 6fn^5 + n^6)\cdot  \\ 27n (f - n)^3 (3f^2 - 6fn + 2n^2)   \\ \Delta=-8503056(n - 2f)(3n - 2f) f^2(2n - f)^2(n - f)^{18}  \end{array} \\ \hline
4 & & \begin{array}{c}A=\tau^3(1+\tau)^2,\ B=0,\\ \Delta=4\tau^9(1+\tau)^6 \end{array} \\ \hline
8 & \begin{array}{l} x_0 = (p - 1)(sp + t)(tp^2 - sp - t) \\ x_1 = s p^2 (tp^2 - sp - t) \\ x_2 = (p - 1) (-sp + tp - t)(sp + t) \end{array} & \begin{array}{c} A=-27t^2s^2(s^4 + 4s^2t^2 + t^4) \\ B = -27t^3s^3(s^2 + 2t^2)(2s^4 + 8s^2t^2 - t^4)  \\ \Delta=8503056s^8 t^{14} (s^2 + 4t^2)  \end{array} \\ \hline
\end{array}
}
\end{equation}
%

The fibration induced by $|N_7|$ is the unique elliptic fibration
on a K3 surface with a fiber of type $I_{14}^*$ (which is a
maximal fiber, since if a K3 surface admits an elliptic fibration
with this reducible fiber, then this is the unique reducible
fiber). So it suffices to know the equation of this elliptic
fibration, which is classically known \cite[Theorem 1.2]{Shioda}. Hence for
$|N_7|$ we only re-write here the known equation, a part from the
fact that we chose the parameters over
$\mathbb{P}^1_{(\tau:\sigma)}$ in such a way that 4 fibers of type
$I_1$ are over the points $(\tau:\sigma)=((\pm\sqrt{-26\pm
14i\sqrt{7}})/4:1)$. These values correspond to the septic in
$|N_7|$ which are tangent to $m_2$ in a smooth point.  We therefore have equation:

\begin{equation}\label{eq: N7}
\begin{array}{c}
A=\tau^2(-12\sigma^6-(3/8)\sigma^4\tau^2-(15/2048)\sigma^2\tau^4-(9/262144)\tau^6),\\
B=\tau^3\sigma(-16\sigma^8-(3/4)\tau^2\sigma^6-(21/1024)\sigma^4\tau^4-(35/131072)\tau^6\sigma^2-(63/33554432)\tau^8),\\
\Delta=-(729/4503599627370496)\tau^{20}(2048\sigma^4+52\tau^2\sigma^2+\tau^4).
\end{array}
\end{equation}

We note that it is possible to obtain such an equation using our techniques.  Indeed one may find a rational parameterization of the septic plane curves of the generalized conic bundle by rational cubic curves, and from there the computation follows in the same way as above.  We omit this computation here as the equation is already in the literature.

\section{The K3 surface $X_{5,5}$ and its elliptic fibrations}\label{sec: the K3 surface X}

The K3 surface $X_{5,5}$ is very well-known and studied, in
particular since its N\'eron--Severi group allows to describe the
moduli spaces of K3 surfaces with an elliptic fibration with a
5-torsion section in terms of $L$-polarized K3 surfaces, see
\cite{G}. Indeed, if a K3 surface admits an elliptic fibration
with a 5-torsion section, then the lattice $U\oplus M$ has to be
primitively embedded in its N\'eron--Severi lattice, where $M$ is
an overlattice of index 5 of a root lattice. The lattice $M$ is
known to be an overlattice of index 5 of $A_4^4$ and so the
lattice $U\oplus M$ is isometric to the N\'eron--Severi lattice of
$X_{5,5}$. Hence the moduli space of the K3 surfaces which are
$NS(X_{5,5})$-polarized coincides with the moduli space of the K3
surfaces which admit an elliptic fibration with a 5-torsion
section.

\subsection{The list of all the elliptic fibrations}
The transcendental lattice of $X_{5,5}$ is known to be $T_X\simeq
U\oplus U(5)$, see e.g. \cite{GSarti}. This allows to apply the Nishiyama
method in order to classify (at least lattice theoretically) the
elliptic fibrations on $X_{5,5}$. This method is studied and
applied in several papers, and we do not intend to describe it in
details. Here we just observe that we can apply the method using
$A_4\oplus A_4$ as the lattice $T$, so that $T$ is a negative-definite
lattice with the  same discriminant form as the one of
the transcendental lattice $T_X$ and whose rank is $\rk(T) = \rk(T_X) + 4$.
Then one has to find the primitive
embeddings of $A_4\oplus A_4$ in the Niemeier lattice up to isometries, and this can be done by embedding $T$
into the root lattice of the Niemeier lattices. The list of the
Niemeier lattices and the possible embeddings of root lattices in
Niemeier lattices up to the Weyl group can be found in
\cite{Nish} (see also \cite{Nie}). In particular one has that $A_4$ embeds primitively
in a unique way (up to the action of the Weyl group) in $A_n$ for
$n>4$, in $D_m$ for $m>4$, in $E_h$ for $h=6,7,8$, see
\cite[Lemmas 4.2 and 4.3]{Nish}. On the other hand $A_4\oplus A_4$
has a primitive embedding in $A_n$ for $n\geq 9$, in $D_m$ for
$m\geq 10$, and has no primitive embeddings in $E_h$ for $h=6,7,8$,
see \cite[Lemma 4.5]{Nish}. All these primitive embeddings are
unique with the exception of $A_4\oplus A_4\hookrightarrow
D_{10}$, for which there are two possible primitive embeddings,
see \cite[Page 325, just before Step 3]{Nish}. The orthogonal complement of
the embedded copy of $A_4\oplus A_4$ in the root lattice of each
Niemeier lattices is a lattice $L$ which can be computed by
\cite[Corollary 4.4]{Nish} and which encodes information about both
the reducible fibers and the rank of the Mordell--Weil group of
the elliptic fibrations. In particular the root lattice of $L$ is
the lattice spanned by the irreducible components of the reducible
fibers orthogonal to the zero section. More precise information on the sections can be obtained
by a deeper analysis of these embeddings, but this is outside the scope of this paper.
So in the following list we give the root lattice of the Niemeier
lattice that we are considering, the embeddings of $A_4\oplus A_4$
in this root lattice, the root lattices of the orthogonal complement, and in
the last two columns the properties of the associated elliptic
fibration. This gives the complete list of the types of elliptic fibrations
on $X_{5,5}$ :
\begin{equation}\label{eq: table elliptic fibrations on X55} \end{equation}
$$\begin{array}{c|c|c|c|c|c|c|c}
\mbox{n}^o&\mbox{Niemeier}&\mbox{embedding(s)}&\mbox{roots orthogonal}&\mbox{singular fibers}&\mbox{rk}(MW)\\
1&E_8^3&A_4\subset E_8\ A_4\subset E_8&A_4\oplus A_4\oplus E_8&2I_5+II^*+2I_1&0\\
2&E_8\oplus D_{16}&A_4\subset E_8\ A_4\subset D_{16}&A_4\oplus D_{11}&I_5+I_{7}^*+6I_1&1\\
3&E_8\oplus D_{16}&A_4\oplus A_4\subset D_{16}&E_8\oplus D_{6}&II^*+I_2^*+6I_1&2\\
4&E_7^2\oplus D_{10}&A_4\subset E_7\ A_4\subset E_7& A_2^2\oplus D_{10}&2I_3+I_6^*+6I_1&2\\
5&E_7^2\oplus D_{10}&A_4\subset E_7\ A_4\subset D_{10}&E_7\oplus A_2\oplus D_{5}&III^*+I_3+I_1^*+5I_1&2\\
6&E_7^2\oplus D_{10}&A_4\oplus A_4\subset D_{10}&E_7^2&2III^*+6I_1&2\\
7&E_7^2\oplus D_{10}&A_4\oplus A_4\subset D_{10}& E_7^2&2III^*+6I_1&2\\
8&E_7\oplus A_{17}&A_4\subset E_7\ A_4\subset A_{17}&A_2\oplus A_{12}&I_3+I_{13}+8I_1&2\\
9&E_7\oplus A_{17}&A_4\oplus A_4\subset A_{17}& E_7\oplus A_7&III^*+I_8+7I_1&2\\
10&D_{24}&A_4\oplus A_4\subset D_{24}& D_{14}&I_{10}^*+8I_1&2\\
11&D_{12}\oplus D_{12}&A_4\subset D_{12}\ A_4\subset D_{12}&D_{7}\oplus D_7&2I_3^*+6I_1&2\\
12&D_{12}\oplus D_{12}&A_4\oplus A_4\subset D_{12}& D_{12}\oplus A_1^2&I_8^*+2I_2+6I_1&2\\
13&D_8^3&A_4\subset D_{8}\ A_4\subset D_{8}&D_{8}\oplus A_3\oplus A_3&I_4^*+2I_4+6I_1&2\\
14&D_9\oplus A_{15}&A_4\subset D_9\ A_4\subset A_{15}&D_4\oplus A_{10}&I_0^*+I_{11}+7I_1&2\\
15&D_9\oplus A_{15}&A_4\oplus A_4\subset A_{15}& D_9\oplus A_{5}&I_5^*+I_6+7I_1&2\\
16&E_6^4&A_4\subset E_6\ A_4\subset E_6&A_1^2\oplus E_6^2&2I_2+2IV^*+4I_1&2\\
17&E_6\oplus D_7\oplus A_{11}&A_4\subset E_6\ A_4\subset D_7&A_1^3\oplus A_{11}&3I_2+I_{12}+6I_1&2\\
18&E_6\oplus D_7\oplus A_{11}&A_4\subset E_6\ A_4\subset A_{11}&A_1\oplus D_7\oplus A_{6}&I_2+I_3^*+I_7+6I_1&2\\
19&E_6\oplus D_7\oplus A_{11}&A_4\subset D_7\ A_4\subset A_{11}& E_6\oplus A_1^2\oplus A_{6}&IV^*+2I_2+I_7+5I_1&2\\
20&E_6\oplus D_7\oplus A_{11}&A_4\oplus A_4\subset A_{11}&E_6\oplus D_7\oplus A_1&IV^*+I_3^*+I_2+5I_1&2\\
21&D_6^4&A_4\subset D_6\ A_4\subset D_6&D_6^2&2I_2^*+8I_1&4\\
22&D_6\oplus A_9^2&A_4\subset D_6\ A_4\subset A_9&A_4\oplus A_9&I_5+I_{10}+9I_1&3\\
23&D_6\oplus A_9^2&A_4\subset A_9\ A_4\subset A_9&D_6\oplus A_4\oplus A_4&I_2^*+2I_5+6I_1&2\\
24&D_6\oplus A_9^2&A_4\oplus A_4\subset A_9&D_6\oplus A_9&I_2^*+I_{10}+6I_1&1\\
25&D_5^2\oplus A_7^2&A_4\subset D_5\ A_4\subset D_5& A_7^2&2I_8+8I_1&2\\
26&D_5^2\oplus A_7^2&A_4\subset D_5\ A_4\subset A_7&A_7\oplus D_5\oplus A_2&I_8+I_1^*+I_3+6I_1&2\\
27&D_5^2\oplus A_7^2&A_4\subset A_7\ A_4\subset A_7&D_5^2\oplus A_2^2&2I_1^*+2I_3+4I_1&2\\
28&A_8^3&A_4\subset A_8\ A_4\subset A_8&A_3^2\oplus A_8&2I_4+I_9+7I_1&2\\
29&A_{24}&A_4\oplus  A_4\subset A_{24}&A_{14}&I_{15}+9I_1&2\\
30&A_{12}^2&A_4\subset A_{12}\ A_4\subset A_{12}&A_7^2&2I_8+8I_1&2\\
31&A_{12}^2&A_4\oplus A_4\subset A_{12}&A_2\oplus A_{12}&I_3+I_{13}+8I_1&2\\
32&D_4\oplus A_5^4&A_4\subset A_5\ A_4\subset A_5&D_4\oplus A_5^2&I_0^*+2I_6+6I_1&2\\
33&A_6^4&A_4\subset A_6\ A_4\subset A_6&A_1^2\oplus A_6^2&2I_2+2I_7+6I_1&2\\
34&A_4^6&A_4\subset A_4\ A_4\subset A_4& A_4^4& 4I_5+4I_1&0
\end{array}
$$
Observe that lines 6 and 7 correspond to the two different embeddings of $A_4\oplus A_4$ in $D_{10}$.
The K3 surface $X_{5,5}$ is obtained as double cover of a rational
surface $R_{5,5}$ branched on two smooth fibers, so there are no
elliptic fibrations induced by generalized conic bundles or by
splitting genus 1 pencils, see \cite{GS}. So an elliptic fibration
on $X_{5,5}$ is either induced by a conic bundle on $R_{5,5}$ or
it is of type 3.

Putting together these considerations with the results of Sections \ref{sec: K3 covers of R55} and  \ref{sec: elliptic fibrations induced by conic bundles}, we proved the following proposition.
\begin{proposition}\label{prop: classification elliptic fibrations on X55} The elliptic fibrations on $X_{5,5}$ are of 34 types, listed in Table \ref{eq: table elliptic fibrations on X55}. The fibration in line 34 of Table \ref{eq: table elliptic fibrations on X55} is induced by $\E_R$ and its equation is given in Section \ref{sub: Weierstrass X55}; the fibrations of lines 22, 30 and 32 are induced by conic bundles on $R_{5,5}$ and their equations are given in Section \ref{subsec: ef induced by cb on X}. The other fibrations on $X_{5,5}$ are of type 3.\end{proposition}

\subsection{Fibration of type 3: an example, the fibration 26}
The aim of this section is to construct explicitly an example of an elliptic fibration of type 3 and to discuss the geometry of the non complete linear system on $\mathbb{P}^2$ which induces this fibration.

The divisor $$D_1:=\Omega_0^{(2,2)}+\Omega_0^{(1,1)}+2Q_0+2\Omega_0^{(2,1)}+\Omega_1^{(2,1)}+\Omega_4^{(2,1)}$$ corresponds to the class of the fiber of a fibration which has one reducible fiber of type $I_1^*$, thus $|D_1|$ is one of the fibrations 5,26,27 in Table \eqref{eq: table elliptic fibrations on X55}. The curves $\Omega_2^{(2,1)}$, $\Omega_3^{(2,1)}$, $Q_2$, $Q_3$, $\Omega_{1}^{(2,2)}$, $\Omega_{4}^{(2,2)}$,  $\Omega_{1}^{(1,1)}$,  $\Omega_{4}^{(1,1)}$ are sections of the fibration $|D_1|$.
Assuming that $\Omega_{2}^{(2,1)}$ is the zero section there is fiber whose non trivial components are $\Omega_{2}^{(2,2)}$, $\Omega_{3}^{(2,2)}$, $Q_4$, $\Omega_{4}^{(1,2)}$, $\Omega_{2}^{(1,2)}$, $\Omega_{3}^{(1,2)}$, $\Omega_{1}^{(1,2)}$. So there is a fiber of type $I_8$ and $|D_1|$ is the fibration 26 in \eqref{eq: table elliptic fibrations on X55} and there is a fiber of type $I_3$ whose non trivial components are $\Omega_{2}^{(1,1)}$, $\Omega_{3}^{(1,1)}$.

Denoted by $\pi:X_{5,5}\ra R_{5,5}$, we have $\pi(Q_i)=P_i$ and $\pi(\Omega_i^{(j,1)})=\pi(\Omega_i^{(j,2)})=\Theta_i^{(j)}$, for $j=1,2$ . In particular $\iota(\Omega_i^{(j,1)})=\Omega_i^{(j,2)}$, thus, $\iota(D_1)\neq D_1$ and indeed $D_2:=\iota(D_1)$ is  $$D_2:=\Omega_0^{(2,1)}+\Omega_0^{(1,2)}+2Q_0+2\Omega_0^{(2,2)}+\Omega_1^{(2,2)}+\Omega_4^{(2,2)}.$$

We observe that $\Omega_0^{(2,1)}D_1=0$,  $\Omega_0^{(1,2)}D_1=2$, $Q_0D_1=0$,
$\Omega_0^{(2,2)}D_1=0$, $\Omega_1^{(2,2)}D_1=1$, $\Omega_4^{(2,2)}D_1=1$, and thus $D_2D_1=0+2+0+0+1+1=4$.

So $D_1D_2=4$ and $(D_1+D_2)^2=8$. In particular a smooth member of the linear system $|D_1+D_2|$ is a curve of genus 5.

We are interested in the class of the curve $\pi(D_1+D_2)$. By the projection formula $\pi_*(D_1+D_2)=2\pi(D_1+D_2)$, so we are looking for $\frac{1}{2}\pi_*(D_1+D_2)$. Since $\pi(D_1)=\pi(D_2)$, $\frac{1}{2}\pi_*(D_1+D_2)=\pi_*(D_1)=\pi_*(D_2)$.

We recall that the map $\pi$ restricted to  $\Omega_i^{(j,1)}$ is a $1:1$ map to $\Theta_i^{(j)}$, and similarly the map $\pi$ restricted to $\Omega_i^{(j,2)}$ is a $1:1$ map to $\Theta_i^{(j)}$. On the other hand the map $\pi$ restricted to the sections $Q_i$ is a $2:1$ map to the section $P_i$. So
$$\pi_*(D_1)=\Theta_0^{(2)}+\Theta_0^{(1)}+4P_0+2\Theta_0^{(2)}+\Theta_1^{(2)}+\Theta_4^{(2)}=3\Theta_0^{(2)}+\Theta_1^{(2)}+\Theta_4^{(2)}+\Theta_0^{(1)}+4P_0.$$

Hence by \eqref{R_5,5 curve ids} $$\pi_*(D_1)=3E_1+\ell_1+\ell_2+m_1+4F_1=3h-E_2-F_2-E_{Q_5}-E_4-2F_4-E_3-F_3,$$
which is the class of the strict transforms of  cubics in $\mathbb{P}^2$ passing through $Q_2$, $Q_3$, $Q_4$ with tangent $\ell_2$, and $Q_5$.

The equation of the cubics satisfying these properties is
\begin{multline}\label{eq: type three}ax_0^3+bx_0^2x_1+(-a-b)x_0x_1^2+dx_0^2x_2+\\ ex_0x_1x_2+fx_1^2x_2+(-3a-2d)x_0x_2^2+(a-e-f)x_1x_2^2+(d+2a)x_2^3=0.\end{multline}

This equation depends on 5 parameters (4 projective parameters) and specializes to equations of cubics which split on the double cover and induces elliptic fibrations on $X_{5,5}$.

The generic cubic $c_3$ as in equation \eqref{eq: type three} is a cubic in $\mathbb{P}^2$. We recall that $X_{5,5}$ is the double cover of $\mathbb{P}^2$ branched along the reducible sextic $c_6$ whose equation is $f_3g_3=0$ for two cubics $f_3$ and $g_3$.
So $c_3$ and $c_6$ meet in 18 points in $\mathbb{P}^2$ counted with multiplicity. Recall that $c_6$ is singular at  $Q_2$, $Q_3$, $Q_4$, and $Q_5$ and in $Q_4$, $c_6$ is the union of two cubics, both with tangent direction $\ell_2$. Hence $c_3$ intersects $c_6$  in $Q_2$, $Q_3$, $Q_5$ with multiplicity 2, and in $Q_4$ with multiplicity 4. Outside these points, $c_3$ and $c_6$ intersect in $18-2-2-2-4=8$ points. The inverse image of $c_3$ on $X_{5,5}$ is a double cover of $c_3$ branched in 8 points. Since generically $c_3$ is a smooth cubic in $\mathbb{P}^2$, it has genus 1 and then its inverse image on $X_{5,5}$ has genus $g$ such that $2g-2=2(0)+8.$ So $g=5$. 

We already observed that $c_6$ is the union of two smooth cubics in
$\mathbb{P}^2$, which are the image of the branch curves of the
double cover $X_{5,5}\ra R_{5,5}$. Denoting by $b_1:f_3=0$ and $b_2:=g_3=0$ these two
cubics,  $c_3$ intersects $b_1$ in nine points, five of which are
$Q_2$, $Q_3$, $Q_4$ with tangent $\ell_2$, and $Q_5$. So $c_3$ and
$b_1$ generically intersect in four other  points. In the case
$c_3$ splits in the double cover (which is the case in which $c_3$
is the image both of $D_1$ and $D_2$), these four points are
either one point with multiplicity 4 or two points with
multiplicity 2. The strict transform of $b_1$ (resp. $b_2$) on
$R_{5,5}$ is a smooth fiber of the fibration on $R_{5,5}$ and its
pullback to $X_{5,5}$ is a smooth fiber, denoted by $B_1$ (resp. $B_2)$\footnote{Not to be confused with the conic bundles $B_1$ and $B_2$ in Section \ref{sec: conic bundles}}, of
the fibration induced on $X_{5,5}$ by the one on $R_{5,5}$. Since a section of this fibration is $Q_0$, $B_1Q_0=B_2Q_0=1$,
$B_1\Omega_i^{(k,j)}=0$ and thus $D_1B_1=D_1B_2=D_2B_1=D_2B_2=2$.
In particular $D_1$ and $D_2$ intersect in four points, two on
$B_1$ and two on $B_2$. Considering the image of these curves in
$\mathbb{P}^2$, one realizes that if $c_3$ is the image of $D_1$,
then it intersects $b_1$ (resp. $b_2$) in two points each with
multiplicity 2.

So, theoretically, in order to find the specializations of a curve $c_3$ which is the image of $D_1$, one has to require that the intersection $c_3\cap b_1$  consists of the points $Q_2$, $Q_3$, $Q_4$ with tangent $\ell_2$, and $Q_5$, and of two other  points each with multiplicity 2.

As in the previous context one can compute the resultant between
the equation of the cubics $c_3$ and the branch locus of the
double cover $X_{5,5}\ra\mathbb{P}^2$, given in \eqref{eq: X55 2:1 cover of  P2}. Since not all the curves in the linear system
$|c_3|$ split in the double cover, the resultant of the equation of $c_3$ and the equation
of the branch locus of  $X\rightarrow\mathbb{P}^2$  is not the
square of a polynomial. Nevertheless one recognizes some factors
with even multiplicity (which correspond to the conditions that
$c_3$ passes with a certain multiplicity through a certain base
point of the pencil of cubics from which  $R_{5,5}$ arises) and
one can also observe that for certain choices of the values
$(a,d,e,f)$  the resultant becomes a square. For examples one
observes that for $f=-a-d-e$ the resultant with respect to $x_1$
is
$$
\begin{array}{c}x_2^2(x_0-x_2)^6(x_0x_2(em+am-a+dm)+x_0^2(am+bm)-x_2^2(d-2a))^2\\(x_0x_2(el+al-a+dl)+x_0^2(al+bl)-x_2^2(d-2a))^2\end{array}$$
so in this case we know that the intersection between the generic member of $|c_3|$ and the branch curve is always with even degree, which is the necessary condition to have a splitting in any point.


\begin{thebibliography}{99}
\bibitem{Beau} A. Beauville {\it Les familles stables de courbes elliptiques sur $\mathbb{P}^1$ admettant 4 fibres singuli\`eres}, C. R. Acad. Sc. Paris 294, 657-660 (1982).
\bibitem{CG} P. Comparin, A. Garbagnati {\it Van Geemen--Sarti
involutions and elliptic fibrations on K3 surfaces double cover of
$\mathbb{P}^2$}, Journal of Mathematical Society of Japan {\bf 66}
(2014) 479-522.
\bibitem{GSarti} A.\ Garbagnati, A.\ Sarti, {\it Symplectic automorphisms of prime order on K3
surfaces}, J. Algebra {\bf 318} (2007), 323--350.
\bibitem{GS} A.\ Garbagnati, C.\ Salgado, {\it Linear systems on rational elliptic surfaces and elliptic fibrations on K3
surfaces}, arXiv:1703.02783.
\bibitem{G} A.\ Garbagnati, {\it Elliptic K3 surfaces with abelian and dihedral groups of symplectic
automorphisms}, Comm. Algebra {\bf 41} (2013)
583--616.
\bibitem{Mi} R. Miranda, {\it The basic theory of elliptic surfaces}, Dip. di matematica-Univ. Pisa, available on line at http://www.math.colostate.edu/~miranda/BTES-Miranda.pdf. R. Miranda
\bibitem{Nie}
H.-V.\ Niemeier, {\it Definite quadratische {F}ormen der {D}imension 24 und {D}iskriminante 1}, Journal of Number Theory, {\bf 5} (1973), 142--178.
\bibitem{Nik}
V.\ V.\ Nikulin, {\it Factor groups of groups of automorphisms of hyperbolic forms with respect to subgroups generated by 2-reflections}, J. Soviet Math., {\bf 22} (1983), 1401--1475.
\bibitem{Nish}
K.\ Nishiyama, {\it The Jacobian fibrations on some K3  surfaces and their Mordell-Weil groups}. Japan. J. Math. (N.S.)  {\bf 22}  (1996), 293--347.
\bibitem{RSS} Q.\ Ren, K.\  Shaw, B.\ Sturmfels, {\it Tropicalization of del Pezzo surfaces},  Adv. Math.  {\bf 300}  (2016), 156--189.
\bibitem{SS}  M. Schuett, T. Shioda, {\it Elliptic surfaces}, Algebraic geometry in East Asia - Seoul 2008, Advanced Studies in Pure Mathematics {\bf 60} (2010), 51--160.
\bibitem{Shimada} I.\ Shimada, {\it On elliptic K3 surfaces}, Michigan Math. J. {\bf 47} (2000), 423--446, arXiv version with the complete Table  arXiv:math/0505140.
\bibitem{Shioda} T.\ Shioda {\it The elliptic K3 surfaces with a maximal singular fibre}, C. R. Acad. Sci.
Paris, Ser. I, 337 (2003), pp. 461--466.
\bibitem{Vin} E.B. Vinberg, {\it The two most algebraic K3 surfaces}, Math. Ann. 265 (1983), 1--21.
\bibitem{Z}  D. Q. Zhang, {\it Quotients of K3
surfaces modulo involutions}, Japan J. Math. (N.S.) {\bf 24}
(1998), 335--366.


\end{thebibliography}
\end{document}